\RequirePackage{fix-cm}
\documentclass[smallcondensed]{mysvjour3}     
\smartqed  
\usepackage{graphicx}
 \usepackage{mathptmx}      
\usepackage{amsmath,amssymb}
\usepackage{graphicx}
\usepackage{caption}
\usepackage{mathtools}
\usepackage{enumerate}
\usepackage[all]{xypic}
\usepackage{verbatim}
\usepackage{chemfig, chemnum}
\usepackage{tikz}
\usepackage[lined,commentsnumbered,ruled]{algorithm2e}
\usepackage{caption}
\usepackage{todonotes}
\usepackage{makecell}
\usepackage{array}
\newcommand{\PP}{\mathbb{P}}
\newcommand{\RR}{\mathbb{R}}
\newcommand{\QQ}{\mathbb{Q}}
\newcommand{\CC}{\mathbb{C} }

\begin{document}

\title{Theta Surfaces
\thanks{Dedicated to J{\"u}rgen Jost on the occasion of his 65th birthday.}}



\author{Daniele Agostini \and T\"urk\"u \"Ozl\"um \c{C}elik \\ Julia Struwe  \and Bernd Sturmfels 
}


\institute{Daniele Agostini \at  
              Humboldt-Universit\"at zu Berlin \\
              \email{daniele.agostini@math.hu-berlin.de}           
           \and
            T\"urk\"u \"Ozl\"um \c{C}elik \at
              Universit\"at Leipzig and MPI-MiS Leipzig \\
              \email{turkuozlum@gmail.com} 
              \and
            Julia Struwe \at
              Universit\"at Leipzig and MPI-MiS Leipzig \\ 
              \email{julia.struwe@posteo.de}
              \and
           Bernd Sturmfels \at
              MPI-MiS Leipzig and UC Berkeley \\ 
              \email{bernd@mis.mpg.de}
}



\maketitle

\begin{abstract}
A theta surface in affine 3-space is the zero set
of a Riemann theta function in genus~3. This includes
surfaces arising from special plane quartics that
are singular or reducible.  
Lie and Poincar\'e showed that any analytic surface that is
the Minkowski sum of two space curves in two different ways is a theta surface.
The four space curves that generate such a
double translation structure
 are parametrized by abelian integrals, so they are usually not algebraic.
This paper offers a new view on this classical topic through the lens of computation.
We present practical tools for passing between quartic curves and their theta surfaces,
and we develop the numerical algebraic geometry of
degenerations of theta functions.

\end{abstract}

\section{Introduction}

Our first example of a theta surface is {\em Scherk's minimal surface}, given by the equation
\begin{equation}
\label{eq:scherk}
{\rm sin}(X)\, \,-\,\, {\rm sin}(Y) \cdot {\rm exp}(Z)\,\,=\,\, 0.
 \end{equation}
 
\begin{figure}[ht]
  \centering
  \includegraphics[scale = 0.45]{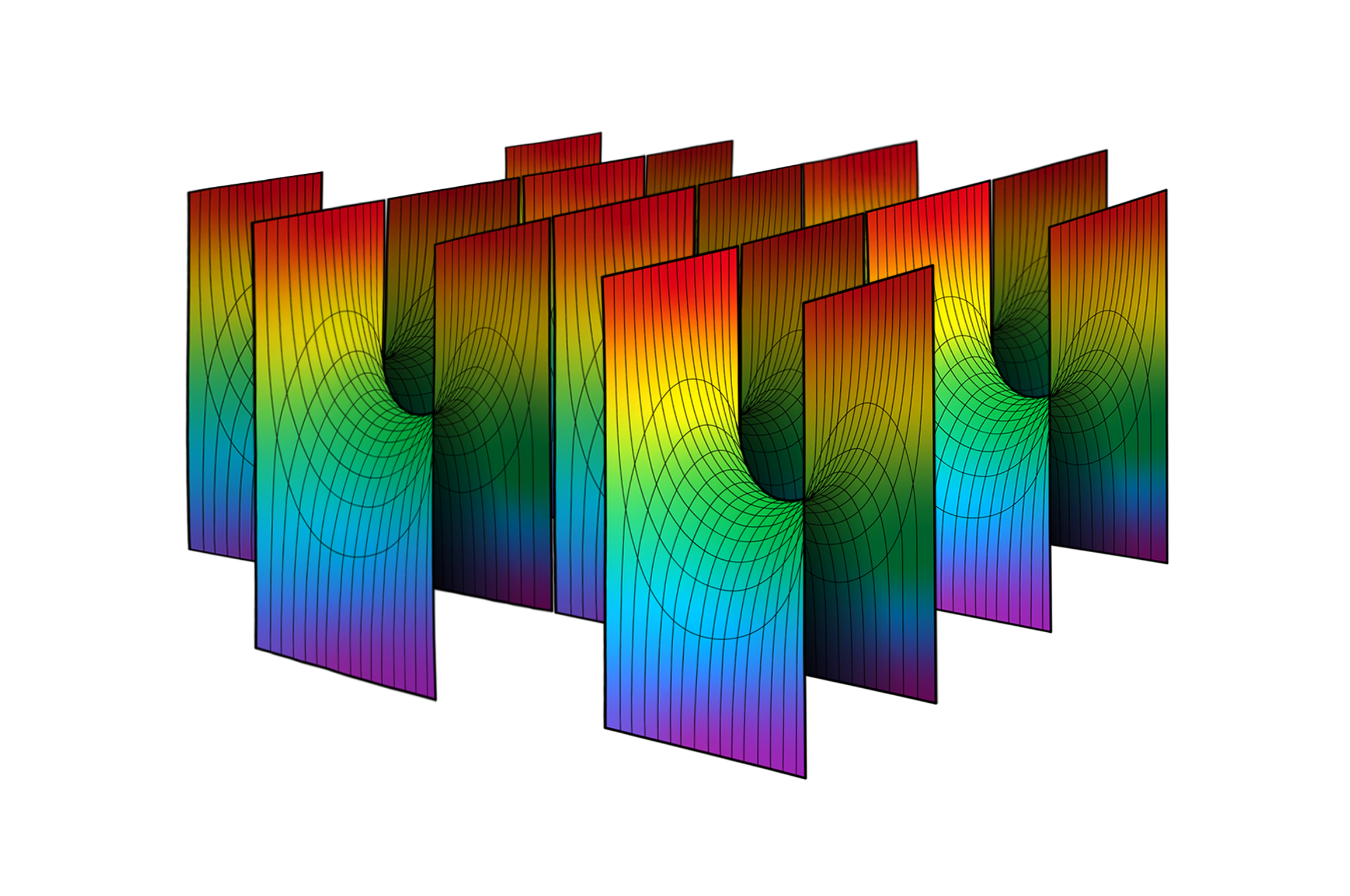}
  \vspace{-0.2in}
  \caption{Scherk's minimal surface.} 
    \label{polytope}
\end{figure}

 This surface arises from the following quartic
 curve in the complex projective plane~$\PP^2$:
 \begin{equation} 
 \label{eq:quartic1}
xy(x^2+y^2+z^2) \,\,=\,\,0. 
 \end{equation}
We use $X,Y,Z$ as affine coordinates for $\RR^3$
and $x,y,z$ as homogeneous coordinates for $\PP^2$.
Scherk's minimal surface is obtained as the Minkowski sum of two parametric space curves:
\begin{equation}
\label{eq:scherkpara1}
\begin{matrix}
\!  (X,Y,Z)  \,\,& \, = & 
\bigl( {\rm arctan}(s), 0, {\rm log}(s) -  {\rm log}(s^2{+}1)/2 \bigr)  \\ & & +\,\,
\bigl( 0,  -{\rm arctan}(t),- {\rm log}(t) + {\rm log}(t^2{+}1)/2 \bigr)  . 
\end{matrix}
\end{equation}
 The derivation of  (\ref{eq:scherk})  from (\ref{eq:quartic1})
via (\ref{eq:scherkpara1}) is given in Example \ref{ex:scherkalgo1}.
This  computation is originally due to Richard~Kummer \cite[p.~52]{Kum} whose
1894 dissertation also displays a plaster model.

Following the classical literature (cf.~\cite{Little83,Sch}), a {\em surface of double translation} equals
\begin{equation*}
 \mathcal{S} \quad = \quad \mathcal{C}_1 + \mathcal{C}_2 \,\, = \,\, \mathcal{C}_3 + \mathcal{C}_4 , 
 \end{equation*}
where $\mathcal{C}_1,\mathcal{C}_2,\mathcal{C}_3,\mathcal{C}_4$ are curves in $\RR^3$,
and the two decompositions are distinct.

We note that Scherk's minimal surface
(\ref{eq:scherk})  is a surface of double translation. A first representation 
$ \mathcal{S} = \mathcal{C}_1 + \mathcal{C}_2$ was given in (\ref{eq:scherkpara1}).
A second representation  $ \mathcal{S} = \mathcal{C}_3 + \mathcal{C}_4$ equals
\begin{equation}
\label{eq:scherkpara2}
\begin{matrix}
X  & = & {\rm arctan}(u)\,\,\,+\,\,\,{\rm arctan} (v) & = & 
{\rm arctan} \bigl(\frac{u+v}{1-uv} \bigr), \medskip \\
Y & = & {\rm arctan}(5u)\,\,+\,\,{\rm arctan} (5v) & = & \,
{\rm arctan} \bigl(\frac{5u+5v}{1-25uv} \bigr),  \smallskip \\
Z & = &   \frac{1}{2} {\rm log}\bigl( \frac{1+(5u)^2}{5(1+u^2)} \bigr) \, + \,
                \frac{1}{2} {\rm log}\bigl( \frac{1+(5v)^2}{5(1+v^2)} \bigr).
                & & 
\end{matrix}
\end{equation}
It is instructive to verify that both parametrizations (\ref{eq:scherkpara1}) and
(\ref{eq:scherkpara2}) satisfy the equation~(\ref{eq:scherk}). 

A remarkable theorem
due to Sophus Lie \cite{LieGes2}, refined by Henri Poincar\'e in \cite{Poin95},
states that these are precisely the surfaces derived from plane quartic curves by
the integrals appearing in Abel's Theorem.
The implicit equation of such a surface is an analytic object introduced by
Bernhard Riemann. Namely, if the quartic is smooth then this is Riemann's
 {\em theta function} $\theta$.
Modern algebraic geometers view the surface
$\{ \theta = 0\}$ as the {\em theta divisor} in the Jacobian.
We replace that abelian threefold by
its universal cover $\CC^3$ and we focus on the subset $\RR^3$ of real points.
Our object of study is
 the real analytic surface
 $\{ \theta = 0\}$  in $\RR^3$.

 The present article is organized as follows.
 Section \ref{sec2} derives the parametrization of theta surfaces by abelian integrals.
 We review Abel's Theorem~\ref{thm:abel},
 Riemann's Theorem~\ref{thm:riemann} and
 Lie's Theorem \ref{thm:lie},  all from the perspective developed by
Lie's successors in \cite{Eies08,Eies09,Kum,Sch,Wie}.
 In Section \ref{sec3} we present a symbolic algorithm for computing theta
 surfaces. Here the input is a reducible quartic curve whose abelian integrals
 can be evaluated in closed form in a computer algebra system. As an illustration we show
  how (\ref{eq:scherkpara1}) and (\ref{eq:scherk}) are derived from (\ref{eq:quartic1}).
 
In Section~\ref{sec5} 
we discuss degenerations of curves and their Jacobians via tropical geometry.
This leads to the formula in Theorem \ref{thm:degtheta} for the implicit equation
   of a degenerate theta surface. 
 Based on the combinatorics of Voronoi cells and Delaunay polytopes, this
 offers a present-day explanation for formulas, such as (\ref{eq:scherk}),
 that were found well over a century ago.

Theta surfaces are usually transcendental, but they can be algebraic in special~situations. Algebraic theta surfaces were classified by John Eiesland in~\cite{Eies08}. 
Examples include the cubic surface mentioned by
Shiing-Shen Chern in \cite[p.~2]{Chern}, the quintic surface
shown by John Little in \cite[Example 4.3]{Little83} and the quadric surface arising from the union of four concurrent lines in Example~\ref{Quadric}.
In Section \ref{sec6} we revisit Eiesland's census of quartics with algebraic theta surfaces. We
present derivations and connections to sigma functions \cite{BEL,naka}.

In Section \ref{sec4} we study theta surfaces via numerical computation.
   Building on
state-of-the-art methods for evaluating abelian integrals and theta functions,
we develop a numerical algorithm whose input is a 
smooth quartic curve in~$\PP^2$ and whose output is its
theta surface.

Our article proposes the name {\em theta surface}
for what Sophie Lie called a {\em surface of double translation}.
We return to historical sources in Section \ref{sec7}, by offering
a retrospective on the  remarkable work done in Leipzig in the late
$19$th century  by Lie's circle  \cite{Lie,Sch}.
 Our presentation here serves to connect differential geometry and algebraic geometry,
reconciling the  $19$th and $20$th centuries, with a view towards applied mathematics in
           the $21$st century.
     On that note, there is a natural connection 
     to integrable systems and mathematical physics. The three-phase solutions \cite{DFS} to 
     the KP equation are closely related
     to  theta surfaces. Double translation surfaces also 
     represent invariants in the study of 4-wave interactions in~\cite{BalFer}. 

 \section{Parametrization by Abelian Integrals}
\label{sec2}

We begin with the parametric representation of theta surfaces.
Our point of departure is a real algebraic curve $\mathcal{Q}$
of degree four in the projective plane $\PP^2$.
This is the zero set of  a ternary quartic
that is unique up to scaling.
We thus identify the curve with its polynomial
\begin{equation*}
\label{eq:quartic2}
\qquad
 Q(x,y,z) \quad =\,\,
 \sum_{i+j+k = 4} \! \!c_{ijk} x^i y^j z^k, \qquad \qquad
 c_{ijk} \in \RR.
 \end{equation*}
This quartic can be reducible, as in (\ref{eq:quartic1}), but we assume that it is squarefree.
The symbol $\mathcal{Q}$ refers to the complex curve, and we write
$\mathcal{Q}_\RR$ for its subset of real points.
If $\mathcal{Q}$ is nonsingular then it is a non-hyperelliptic Riemann surface of genus $3$,
canonically embedded into $\PP^2$.

The holomorphic differentials on $\mathcal{Q}$ comprise a $3$-dimensional vector space $H^0(\mathcal{Q}, \Omega^1_\mathcal{Q})$.  Assuming that $z$ does not divide $Q$,
we choose a basis $\{\omega_1,\omega_2,\omega_3\}$ for this space as follows.
Consider the dehomogenization $q(x,y) = Q(x,y,1)$, set
$q_x = \partial q / \partial x$ and $q_y = \partial q / \partial y$, and fix
\begin{equation}
\label{eq:omegabasis1}
\omega_1 \, = \,\frac{x}{q_y} dx \, , \quad
\omega_2 \, = \,\frac{y}{q_y} dx \, , \quad
\omega_3 \, = \,\frac{1}{q_y} dx.
\end{equation}
As the quartic $q(x,y)$ is defined over $\RR$, so is the basis.
To be precise, all coefficients appearing in the $\omega_i$ are real numbers.
Since $\,d q = q_x dx + q_y dy = 0 \,$ holds on  $\mathcal{Q}$, we can also~write
\begin{equation}
\label{eq:omegabasis2}
\omega_1 \, = \,-\frac{x}{q_x} dy \, , \quad
\omega_2 \, = \,-\frac{y}{q_x} dy \, , \quad
\omega_3 \, = \,-\frac{1}{q_x} dy.
\end{equation}

\begin{remark}
\label{remk:canonicalemb}
If $p=(x:y:1)$ is any point on $\mathcal{Q}$, then 
$\,(\omega_1(p):\omega_2(p):\omega_3(p)) \,=\,  p$.
This reflects the fact that $\mathcal{Q}$ is the canonical embedding of
the abstract curve underlying $\mathcal{Q}$.
\end{remark}

Consider a path on the Riemann surface $\mathcal{Q}$ with end points $p$ and $r$.
We can integrate the form $\omega_i$ along that path and obtain
a complex number $\int_{p}^{r} \omega_i$.
If the path is real,  from $p$ to $r$ on a single connected component of
the real curve
$\mathcal{Q}_\RR$, then $\int_{p}^{r} \omega_i$ is a real number.
This number is computed in practise by regarding $y = y(x)$
as a function of $x$, defined by $q(x,y) = 0$, and by
computing the definite integral from the $x$-coordinate of $p$ to the $x$-coordinate of~$r$.
Alternatively, after replacing (\ref{eq:omegabasis1})  with (\ref{eq:omegabasis2}),
we can regard $x = x(y)$ as an implicitly defined  function of $y$, and take
the definite integral from the $y$-coordinate of $p$ to that of~$r$.

We fix a line $\mathcal{L}(0)$ in $\PP^2$ that intersects  $\mathcal{Q}$ in four
distinct nonsingular points $p_{1}(0)$, $p_{2}(0),p_{3}(0)$ and $p_{4}(0)$. For $p_j \in \mathcal{Q}$ sufficiently close to $p_j(0)$,
 we obtain an analytic curve 
\begin{equation}
\label{eq:omega}
\begin{matrix}
\Omega_j (p_j) \,\,= \,\,(\Omega_{1j},\Omega_{2j},\Omega_{3j}) (p_j) \,\, = \,\,
\bigl(\, \int_{p_j(0)}^{p_j} \omega_1\, ,\, \int_{p_j(0)}^{p_j} \omega_2 \,,\,  \int_{p_j(0)}^{p_j} \omega_3 \, \bigr).
\end{matrix}
\end{equation} 

\begin{remark}
\label{rmk:omegatangents} 
By the Fundamental Theorem of Calculus, the derivatives of these curves are
\begin{equation*}
\label{eq:omegaderivative}
\dot{\Omega}_j(p_j) \,\,\,= \,\,\,(\omega_1(p_j),\omega_2(p_j),\omega_3(p_j))\,\,\, = \,\,\, p_j . 
\end{equation*}
Here, the derivative is performed with 
respect to the local parameter of the point $p_j$. The second equality is Remark \ref{remk:canonicalemb}. In words,
the tangent direction to $\Omega_j$ at $p_j$ is $p_j$ itself.
\end{remark}  

The following theorem paraphrases a basic result from the theory of Riemann surfaces:

\begin{theorem}[Abel's Theorem] \label{thm:abel}
Suppose that the points $p_1,p_2,p_3,p_4$ are collinear. Then
\begin{equation*}
\Omega_1(p_1)\,+\,\Omega_2(p_2)\,+\,\Omega_3(p_3)\,+\,\Omega_4(p_4) \,\,\,=\,\,\, 0.
\end{equation*} 
\end{theorem}

Now, let us fix analytic coordinates $s,t$ on the curve, centered at $p_1(0)$ and $p_2(0)$ respectively. Then,
 for every choice of $s,t$ in a small neighborhood of $0$, we obtain two nearby points $p_1(s),p_2(t)$.
  The corresponding theta surface $\mathcal{S}$ is the image of the parametrization
\begin{equation}
\label{eq:stmap}
(s,t) \quad \mapsto \quad \Omega_1(p_1(s)) \,+ \, \Omega_2(p_2(t)) .
\end{equation}
The image is a complex analytic surface $\mathcal{S}$ in $\CC^3$. 
However, if the points $p_1(0)$ and $p_2(0)$ are real
then we take $s$ and $t$ in a small neighborhood of $0$ in  $\RR$, and 
$\mathcal{S}$ is a real surface  in~$\RR^3$.
This surface is analytic and only defined locally, since $s$ and $t$ are local parameters.

Shifting gears, let us now consider an arbitrary analytic
surface $\mathcal{T}$ in $\mathbb{C}^3$. We say that $\mathcal{T}$ is
 a \emph{translation surface}
 if there are two smooth analytic curves $\mathcal{C}_1,\mathcal{C}_2\subset \mathbb{C}^3$ such that 
	\[ \mathcal{T} \,\,=\,\, \mathcal{C}_1 \,+\, \mathcal{C}_2  
	\,\,=\,\,\bigl\{\, p_1 + p_2 \,\,|\,\, p_1\in \mathcal{C}_1, p_2\in \mathcal{C}_2 \,\bigr\} .\]
In words, $\mathcal{T}$ is the {\em Minkowski sum} of 
the two {\em generating curves} $\mathcal{C}_1$ and $ \mathcal{C}_2$.
	 We require the parametrization to be injective, 
i.e.~for each point $x\in \mathcal{T}$ there are unique points 
$p_1\in \mathcal{C}_1,p_2\in \mathcal{C}_2$ such that $x=p_1+p_2$.  
Note that,
 if $\alpha_1,\alpha_2\colon \Delta \to \mathbb{C}^3$ are local parametrizations of 
 the generating curves $\mathcal{C}_1,\mathcal{C}_2$, then the translation surface $S$ has the parametrization
   \[ \mathcal{T} \,\,= \,\,\{ \alpha_1(s) + \alpha_2(t) \} \,\,=\,\, 
   \left\{ \begin{pmatrix} \alpha_{11}(s)+\alpha_{21}(t) \\ \alpha_{12}(s)+\alpha_{22}(t) \\ \alpha_{13}(s) + \alpha_{23}(t) \end{pmatrix} \right\} .\]
        
   \begin{definition}[Double translation surfaces] A translation surface $\mathcal{T}\subset \mathbb{C}^3$ as above is a \emph{double translation surface} if there exists other smooth analytic curves $\mathcal{C}_3,\mathcal{C}_4 \subset \mathbb{C}^3$ such that
\begin{equation} \label{eq:doubletranslation}
   	  \mathcal{T} \,\,=\,\, \mathcal{C}_1\,+\,\mathcal{C}_2 \,\,= \,\,\mathcal{C}_3\,+\,\mathcal{C}_4. 
\end{equation}	  
\end{definition}

Returning to the setting of algebraic geometry, let $\mathcal{S}$
be the theta surface derived as above from a quartic curve $\mathcal{Q}$
in the plane $\PP^2$.
 This formula (\ref{eq:stmap}) shows that $\mathcal{S}$ is a translation surface.
In fact,          the theta surface $\mathcal{S}$ is a double translation surface. 
This is a consequence of Abel's Theorem. To see this, consider the line
 $\mathcal{L}$ that is spanned by the points $p_1(s)$ and $p_2(t)$ in $\PP^2$.
This line intersects the quartic curve $\mathcal{Q}$ in two other
points $p_3(s,t)$ and $p_4(s,t)$, and these points are close to $p_3(0)$ and $p_4(0)$ respectively.
Theorem \ref{thm:abel} says that
\begin{equation}
\label{eq:stmapabel}
\Omega_1(p_1(s))\,+\,\Omega_2(p_2(t))\,\, =\,\, -\Omega_3(p_3(s,t))\, -\, \Omega_4(p_4(s,t)).
\end{equation}
The points $p_3(s,t)$ and $p_4(s,t)$ span the same line $\mathcal{L}$, so they determine $p_1(s)$
and $p_2(t)$ as the residual intersection points of the curve $\mathcal{Q}$ with $\mathcal{L}$. This means  that the points $p_3(s,t)$ and $p_4(s,t)$ can move freely in neighborhoods of $p_3(0)$ and $p_4(0)$ on the curve
$\mathcal{Q}$. If $u,v$ are analytic coordinates on $\mathcal{Q}$, 
centered in $p_3(0)$ and $p_4(0)$ respectively, then  \eqref{eq:stmapabel} shows that 
\begin{equation*}
\label{eq:uvmaps}
(u,v) \quad \mapsto \quad  -\Omega_3(p_3(u))-\Omega_4(p_4(v))
\end{equation*}
is another parametrization of the surface $\mathcal{S}$.
Hence (\ref{eq:doubletranslation}) holds, with generating curves
\begin{equation}\label{eq:analyticcurves}
\mathcal{C}_1 = \Omega_1(p_1(s)), \quad \mathcal{C}_2 =\Omega_2(p_2(t)), \qquad \mathcal{C}_3 = -\Omega_3(p_3(u)), \quad 
\mathcal{C}_4 = -\Omega_4(p_4(v)).
\end{equation}
In particular, we see that the two translation structures on 
the
 theta 
surface 
$\mathcal{S}$
 are distinct. Indeed, Remark \ref{rmk:omegatangents} tells us that the tangent lines to these curves at $0$ correspond to the four points $p_1(0),p_2(0),p_3(0),p_4(0)\in \PP^2$, and these are distinct by construction.

\begin{remark}\label{rmk:quarticfromcurves}
Remark \ref{rmk:omegatangents} shows that 
the tangent directions to the analytic curves in \eqref{eq:analyticcurves}~are
\begin{align*} 
\dot{\mathcal{C}}_1(s) &= \phantom{-} \dot{\Omega}_1(p_1(s)) = p_1(s),  
& \dot{\mathcal{C}}_2(t) &= \phantom{-} \dot{\Omega}_2(p_2(t))=p_2(t),\\
 \dot{\mathcal{C}}_3(u) &= -\dot{\Omega}_3(p_3(u)) = p_3(u), &  
\dot{\mathcal{C}}_4(v) &= -\dot{\Omega}_4(p_4(v)) = p_4(v) .
\end{align*}
Here $p_1(s),p_2(t),p_3(u),p_4(v)$ are regarded as points in the projective 
plane $\mathbb{P}^2$, indicating tangent directions in $\CC^3$, so the sign does not matter. 
Hence, as in Remark \ref{remk:canonicalemb},
 the analytic arcs $\dot{\mathcal{C}}_1,\dot{\mathcal{C}}_2,\dot{\mathcal{C}}_3,\dot{\mathcal{C}}_4$ lie on the quartic $\mathcal{Q}$. In particular, if we are given a theta surface $\mathcal{S}$ and one generating curve $\mathcal{C}_i$, but not the quartic $\mathcal{Q}$, then  we can recover $\mathcal{Q}$  as a quartic in $\PP^2$ that contains the analytic arc $\dot{\mathcal{C}}_i$.
 This fact will be used in Section \ref{sec7}, together with a result of Lie, to give a differential-geometric solution 
to Torelli's problem for genus three curves.
\end{remark}

\smallskip 

Now assume that $\mathcal{Q}$ is a smooth quartic curve. Then we can find an implicit equation for our 
surface via the \emph{Riemann theta function}. Recall \cite{Fay,Mum1} that this is the holomorphic function
\begin{equation}
\label{eq:theta}
\theta(\mathbf{x},B) \,\,= \,\,
\sum_{n\in \mathbb{Z}^3}\mathbf{e}\left( -\frac{1}{2}n^tBn \,+\,i\, n^t\mathbf{x} \right) 
\,\, =\,\,
  \sum_{n\in \mathbb{Z}^3}\exp\left(-\pi n^tBn\right)\cdot \operatorname{cos}(2\pi n^t\mathbf{x}),
\end{equation}
where $\mathbf{e}(t):=\exp(2\pi t)$, $\mathbf{x}=(X,Y,Z)\in \mathbb{C}^3$, and
$B \in \mathbb{C}^{3\times 3}$ is a symmetric matrix with positive-definite real part.
The second sum shows that $\theta$ takes real values if $B$ and $\mathbf{x}$ are~real. 

This definition is slightly different from the usual ones, where $B$ is taken either with
positive definite imaginary part or with negative definite real part. We choose this version
of the Riemann theta function in order to highlight the real numbers. Namely,
 the \emph{theta divisor}
\begin{equation}
\label{eq:thetadivisor}
\Theta_B \,\,:=\,\, \{ \,\mathbf{x} \,\,|\,\, \theta(\mathbf{x},B) = 0\, \}\, \,\,\subset \,\, \,\mathbb{C}^3.
\end{equation}
restricts to a real analytic surface when both $\mathbf{x}$ and $B$ are real.  In general, we will show that the theta divisor coincides with the theta surface above. To this end, we choose a symplectic basis $\alpha_1,\beta_1,\alpha_2,\beta_2,\alpha_3,\beta_3$ for $H_1(\mathcal{Q},\mathbb{Z})$.
The intersection product on $\mathcal{Q}$ is given~by
\begin{equation*}\label{symplecticBasis}
(\alpha_j \cdot \alpha_k) =0, \qquad (\beta_j \cdot \beta_k )=0, \quad {\rm and} \quad (\alpha_j \cdot \beta_k) = \delta_{jk}.
\end{equation*} 
The {\em period matrix} for the Riemann surface $\mathcal{Q}$ with respect to this 
basis equals
\begin{align}\label{periodMatrix}
\Pi \,\,= \,\,\left( \Pi_{\alpha} \,|\, \Pi_{\beta} \right) \,\,\in\,\, \mathbb{C}^{3\times 6},
\end{align}
with entries $(\Pi_{\alpha})_{jk} = \int_{\alpha_k} \omega_j$ and $(\Pi_{\beta})_{jk} = \int_{\beta_k}\omega_j$. As a consequence of Riemann's relations \cite[Theorem 2.1]{Mum1}, the matrix $\Pi_{\alpha}$ is invertible, 
and the corresponding \emph{Riemann matrix} is
\begin{equation}
\label{eq:riemannmatrix}
B \,\,=\,\, -i \cdot  \Pi_{\alpha}^{-1}\Pi_{\beta} .
\end{equation}
Riemann's relations  also show that the matrix $B$ is symmetric with positive-definite real part, so we can consider  the theta divisor $\,\Theta_B \subset \mathbb{C}^3\,$ as in
(\ref{eq:thetadivisor}).
A fundamental theorem of Riemann implies that this coincides with our theta surface, up to a change of coordinates.

\begin{theorem}[Riemann's Theorem] \label{thm:riemann}
The theta surface $\mathcal{S}$ and the theta divisor $\Theta_B$ coincide up to an affine change 
of coordinates on $\mathbb{C}^3$. More precisely, there is a vector $c\in \mathbb{C}^3$ such~that 
	\begin{equation}
	\mathcal{S} \,\,= \,\,\Pi_{\alpha} \cdot \Theta_B + c,
	\end{equation}
	where the equality is meant on all points where the parametrized surface $\mathcal{S}$ is defined. 
\end{theorem}

\begin{proof}
We outline how to obtain this result from the usual statement of Riemann's Theorem. 
Let $(\eta_1,\eta_2,\eta_3)$ be the basis of $H^0(\mathcal{Q},\Omega^1_{\mathcal{Q}})$ 
obtained from $(\omega_1,\omega_2,\omega_3)$ by 
coordinate change with the matrix $\Pi_{\alpha}^{-1}$. Then $\int_{\alpha_j}\eta_k = \delta_{kj}$ 
and $\int_{\beta_j}\eta_k = i \cdot B_{kj}$. Fix a point $r \in \mathcal{Q}$,
paths from $r$ to $p_1(0)$ and from $r$ to $p_2(0)$, and
 local coordinates $s,t$ on $\mathcal{Q}$ around $p_1(0)$ and $p_2(0)$.
For $s,t$ small, we consider paths from $r$ to $p_1(s)$ and from $r$ to $p_2(t)$. This gives an analytic~map
\begin{equation*} \label{eq:Omega1Omega2}
(s,t) \,\,\,\mapsto \,\,\,\begin{pmatrix} \int_{r}^{p_1(s)} \eta_1 \\ \int_{r}^{p_1(s)} \eta_2 \\ \int_{r}^{p_1(s)} \eta_3  \end{pmatrix}
\, +\, \begin{pmatrix} \int_{r}^{p_2(t)} \eta_1 \\ \int_{r}^{p_2(t)} \eta_2 \\ \int_{r}^{p_2(t)} \eta_3 \end{pmatrix}.
\end{equation*}
The familiar form of Riemann's theorem \cite[Theorem 3.1]{Mum1} shows that
 there is a constant $\kappa \in \mathbb{C}^3$ such that the image of this 
 map coincides with $\Theta_{B}-\kappa$. Now, it is enough to write 
\begin{equation*}
\int_{r}^{p_1(s)} \!\! \! \eta_j \,\,=\,\, \int_r^{p_1(0)}\!\!\! \eta_j \,\,+\,\, \int_{p_1(0)}^{p_1(s)}\!\!\!\eta_j \qquad 
{\rm and} \qquad \int_{r}^{p_2(t)}\!\!\! \eta_j \,\,=\,\, \int_r^{p_2(0)}\!\!\! \eta_j\,\, +\,\, \int_{p_2(0)}^{p_2(t)}\!\!\! \eta_j,
\end{equation*}
and then change the coordinates from the basis $(\hspace{-0.5mm}\eta_1,\hspace{-0.5mm}\eta_2,\hspace{-0.5mm}\eta_3\hspace{-0.5mm})$ back to the basis $(\hspace{-0.5mm}\omega_1,\hspace{-0.5mm}\omega_2,\hspace{-0.5mm}\omega_3\hspace{-0.5mm})$. 
\end{proof}

Our discussion shows that each theta surface is also a surface of double translation, and the Riemann theta function provides an implicit equation when the quartic $\mathcal{Q}$ is smooth. A fundamental result of Lie 
states that all \emph{nondegenerate} surfaces of double translation whose two parametrizations are \emph{distinct} arise in this way. 
Here {nondegenerate} and distinct are technical conditions.
The precise definition, phrased in modern language, can be found in~\cite[Definition 2.2]{Little83}. In particular, the nondegeneracy hypothesis assures that none of the generating curves 
can be a line. This rules out special surfaces such as cylinders or planes.

We recall briefly Lie's construction. Let
$\mathcal{S}=\mathcal{C}_1+\mathcal{C}_2=\mathcal{C}_3+\mathcal{C}_4$ be any surface of double translation in $\mathbb{C}^3$. Then we can identify the tangent lines to the curves $\mathcal{C}_j$ with points in $\mathbb{P}^2$, and taking all these tangent lines we obtain analytic arcs 
$\dot{\mathcal{C}}_1,\dot{\mathcal{C}}_2,\dot{\mathcal{C}}_3,\dot{\mathcal{C}}_4 \subset \mathbb{P}^2$.  If $\mathcal{S}$ is a theta surface then, by Remark \ref{rmk:quarticfromcurves}, all these arcs lie on a common quartic curve $\mathcal{Q}$. Lie proved that this property holds for all
nondegenerate surfaces of double translation~$\mathcal{S}$.

\begin{theorem}[Lie's Theorem] \label{thm:lie}
The arcs	$\,\dot{\mathcal{C}}_1,\dot{\mathcal{C}}_2,\dot{\mathcal{C}}_3,\dot{\mathcal{C}}_4\,$  lie on a common
reduced quartic curve $\mathcal{Q}$
in the projective plane $\PP^2$, and $\mathcal{S} $
 coincides with the theta surface associated to~$\mathcal{Q}$.
\end{theorem}  

Lie's original proof \cite{Lie} involves a complicated system of differential equations 
satisfied by the parametrizations of the curves $\mathcal{C}_j$. These differential equations force the arcs 
$\dot{\mathcal{C}}_j$ to lie on a quartic.  A much simpler proof was subsequently given by Darboux \cite{Darb}.
A modern exposition is found in Little's paper \cite{Little83}, together with generalizations to higher dimensions.   

\section{Symbolic Computations for Special Quartics}
\label{sec3}

There is a fundamental dichotomy in the study of theta surfaces, depending on the
nature of the underlying quartic in $\PP^2$. If it is a smooth quartic
then Riemann's Theorem \ref{thm:riemann} furnishes the defining equation
of the theta surface.
This is the case to be studied numerically in Section \ref{sec4}.
At the other end of
the spectrum are the singular quartics considered in the classical literature.
Here methods from computer algebra can be used to compute the theta surface. This is our
topic in the current section. Our focus lies on evaluating the 
 abelian integrals in~\eqref{eq:omega} by exact symbolic computations, 
 as opposed to numerical evaluations of the integrals.
 
We begin by explaining these methods for our running example from the Introduction.

\begin{example}[Scherk's minimal surface] \label{ex:scherkalgo1}
We here derive \eqref{eq:scherkpara1} from \eqref{eq:quartic1}. This serves as a first
illustration for Algorithm \ref{alg:thetaSurface2} below.
We start with the quartic $q(x,y)=xy(x^2+y^2+1)$. This is the
dehomogenization of \eqref{eq:quartic1} with respect to $z$. Using the
partial derivatives $q_y(x,y)=x^3+3xy^2+x$ and $q_x(x,y)=3x^2y + y^3 + y$,
we compute the differential forms  in
(\ref{eq:omegabasis1}) and (\ref{eq:omegabasis2}).

We choose to evaluate the integrals in (\ref{eq:omega})
over the lines  $y=0$ and $x=0$. These will give us the two  summands in 
the parametrization (\ref{eq:stmap}). The first summand 
is obtained by setting $x=s$ and $y=0$, so that $q_y(s,0)=s(s^2+1)$,
and by computing the antiderivatives of the resulting specialized forms $\omega_j(s,0)$
for $j=1,2,3$. The three coordinates of $\Omega_1(p_1(s))$ are
$$ \begin{matrix}
\Omega_{11}(p_1(s)) & =
&\int\frac{s}{s(s^2+1)}d s &=& \arctan(s), \smallskip \\
\Omega_{21}(p_1(s)) &  =
& \int \frac{0}{s(s^2+1)}d s & =&  0,  \smallskip \\
\Omega_{31}(p_1(s)) & =
& \int \frac{1}{s(s^2+1)} d s & =&  \log(s) - \frac{1}{2}\log(s^2+1).
\end{matrix}
$$
The second summand is obtained by integrating over the line $x=0$,
with parameter $y=t$, so that $q_x(0,t)=t(t^2+1)$ in (\ref{eq:omegabasis2}).
The three coordinates of $\Omega_2(p_2(t))$ are found to be
$$ \begin{matrix}
\Omega_{12}(p_2(t)) & =
& - \int \frac{0}{t(t^2+1)}d t & = &  0, \smallskip \\
\Omega_{22}(p_2(t)) & =
& - \int\frac{t}{t(t^2+1)}d t & = &  -\arctan(t), \smallskip  \\
\Omega_{32}(p_2(t)) & =
&- \int\frac{1}{t(t^2+1)}d t  & = &  -\log (t)+\frac{1}{2}\log(t^2+1).
\end{matrix}
$$
By adding these integrals, we obtain the parametrization of the corresponding theta surface:
\begin{equation}
\label{eq:scherkparara}
X \,=\, \arctan(s)\,,\quad
 Y\,=\, -\arctan(t)\,,\quad
Z\,=\, \log \Big(\frac{s}{\sqrt{s^2+1}}\Big) -\log \Big(\frac{t}{\sqrt{t^2+1}}\Big).
\end{equation}
This is  Scherk's minimal surface (\ref{eq:scherkpara1}). Trigonometry now yields
the implicit equation  in (\ref{eq:scherk}).
\end{example}

\smallskip

The following algorithm summarizes the steps we have performed in Example \ref {ex:scherkalgo1}.
Starting from the quartic curve in $\PP^2$, we compute the two generating curves of
the associated theta surface in $\CC^3$. This is done by evaluating the 
integrals in (\ref{eq:omega}) as explicitly as possible.

 \medskip

 \begin{algorithm}[H]\label{alg:thetaSurface2}
    \KwIn{ A quartic equation $q(x,y)$ describing a reduced 
plane quartic curve.}
    \KwOut{The parametrization (\ref{eq:stmap}) of the theta surface $\mathcal{S}$ in affine $3$-space.}
     \KwSty{Step 1:} Specify two points $p_1$ and $p_2$ on the quartic. \\
	\KwSty{Step 2:} Fix local parameters $(x_1,y_1)$ and $(x_2,y_2)$ 
	 around $p_1$ and $p_2$ respectively. \\
          \KwSty{Step 3:} Write $y_j$ as an algebraic function in $x_j$ on its branch. \\
	 \KwSty{Step 4:} Compute the partial derivative $q_y$ on the two branches. \\
    \KwSty{Step 5:} Substitute $x_1=s$ and $x_2=t$ into the differential forms  
    $\omega_1,\omega_2,\omega_3$ in (\ref{eq:omegabasis1}). \\
    \KwSty{Step 6:} By integrating these differential forms,    compute the vectors 
    \begin{equation*} 
     \begin{matrix}
\Omega_1(p_1(s)) & = &   (\int \frac{s}{q_y(s,y_1(s))} ds , \int \frac{y_1(s)}{q_y(s,y_1(s))} ds , \int \frac{1}{q_y(s,y_1(s))} ds),  \\ 
\Omega_2(p_2(t)) & = &
     (\int \frac{t}{q_y(t,y_2(t))} dt,\int \frac{y_2(t)}{q_y(t,y_2(t))} dt,\int \frac{1}{q_y(t,y_2(t))} dt).
     \end{matrix}
\end{equation*}
\hspace{-1.8mm}
         \KwSty{Step 7:} 
         Output the sum $\Omega_1(p_1(s))+\Omega_2(p_2(t))$ of the generating curves as in~(\ref{eq:stmap}).
    \caption{Computing the parametrized theta surface from its plane quartic}
\end{algorithm}  

\medskip

One important part of Algorithm \ref{alg:thetaSurface2}
is Step 3, where $y_j$ is 
represented as an algebraic function in $x_j$. 
This function has algebraic degree at most four,
so it can be written in radicals. After all, the steps above  are meant as a symbolic algorithm.
However, we found the representation in radicals to be 
infeasible for practical computations unless the quartic is very special. 
Even more crucial is the computation of the indefinite integrals in Step 6. This can be done explicitly whenever the quartic is reducible and all the components are
rational: in that special case, the holomorphic differentials of (\ref{eq:omegabasis1}) restrict to a differential $\frac{f(t)}{g(t)}dt$ on $\mathbb{P}^1$ where $f(t),g(t)$ are polynomials, and any such expression can be integrated symbolically.

In what follows we focus on instances where the quartic $q(x,y)$ is reducible.
 A reducible plane quartic is one of the followings: four straight lines, 
 a conic and two straight lines, two conics, or a cubic and a line. In the sequel, we show the computations of
  such theta surfaces using Algorithm \ref{alg:thetaSurface2}.
   The first three cases were worked out by Richard Kummer~\cite{Kum} in his thesis, and the latter case was
   studied by Georg Wiegner~\cite{Wie}. We here present three further examples of non-algebraic theta surfaces. The algebraic ones will be discussed in Section~\ref{sec6}. 

Consider the first three of the four cases above. Then $q(x,y)$ factors into two conics.
The two conics meet in four points in $\PP^2$. We consider the pencil of conics
through these~points. 

\begin{remark} \label{rmk:onsameconic}
A result due to Lie  states that 
the theta surface $\mathcal{S}$
for a product of two conics only depends on their pencil, provided
$p_1$ and $p_2$  lie on the same conic (cf.~\cite[Section 3]{Kum}).
Hence $\mathcal{S}$ has infinitely many distinct 
representations $\mathcal{C}_1+\mathcal{C}_2$ as a translation surface.
We shall return to this topic  in Theorem \ref{thm:lietetra},
 where it is shown how to compute these representations.
\end{remark}

For instance, for Scherk's surface in Example~\ref{ex:scherkalgo1},
the four points in $\PP^2$ are $(i:0:1),(-i:0:1),(0:i:1),(0:-i:1)$, 
and the pencil is generated by $xy$ and $x^2+y^2+z^2$.
We can replace these two quadrics by any other pair in the pencil
and obtain two generating curves.

We now examine another case which is similar. It will lead to the theta surface in 
(\ref{eq:fourterms}).

\begin{example} \label{eq:amoebaofline}
	Fix the four points $(1:0:0),(0:1:0),(0:0:1),(1:1:1)$ in $\PP^2$. Their pencil of conics
	is generated by $y(x-z)$ and $(x-y)z$. We multiply the first conic with a general member of the pencil to get
	the quartic $Q=y(x-z) \cdot ((1+\lambda )xy-xz-\lambda yz)$.  Here $\lambda $ is a parameter,
	which we included in order to illustrate Remark \ref{rmk:onsameconic}.
	Dehomogenizing $Q$, we get $q=y(x-1)((1+\lambda)xy-x-\lambda y)$. We compute
	$\omega_1,\omega_2,\omega_3$ from the partial derivatives 
	\begin{align*}
	&q_x\,=\,y((1+\lambda)xy-x-\lambda y)+y(x-1)((1+\lambda)y-1)),& \\
	&q_y\,=\,(x-1)((1+\lambda)xy-x-\lambda y)+y(x-1)((1+\lambda)x-\lambda).&
	\end{align*}
	We integrate over the two lines given by $y(x-1)=0$. On the first line
	$\{y=0\}$, we have $q_y=-x(x-1)$. The indefinite forms of the three abelian integrals
	in $\Omega_1(p_1(s))$  are
	$$ \begin{matrix} \int\frac{s}{-s(s-1)}d s \,=\, \log \left(\frac{1}{s-1}\right)\,,\quad
	\int \frac{0}{-s(s-1)}d s\,=\, 0  \,,\quad
	\int \frac{1}{-s(s-1)} d s\,=\, \log \bigl(\frac{s}{s-1} \bigr). 
	\end{matrix}
	$$
	For the line $\{x-1=0\}$ we transform the differentials by
	passing from  \eqref{eq:omegabasis1} to \eqref{eq:omegabasis2}. We~find
	$$ \begin{matrix}
	\int \frac{1}{-t(t-1)}d t = \log \bigl( \frac{t}{t-1} \bigr) \,,\quad
	\int \frac{t}{-t(t-1)}d t= \log\left(\frac{1}{t-1}\right)  \, , \quad
	\int \frac{1}{-t(t-1)}d t = \log \bigl( \frac{t}{t-1} \bigr).
	\end{matrix}
	$$
	We conclude that the resulting theta surface has the parametric representation
	\begin{equation}\label{thetaSurfex1}
	\begin{matrix}
	&X & = &  \log \left(\frac{1}{s-1}\right)\,+\,\log \bigl( \frac{t}{t-1} \bigr),\smallskip \\
	& Y & = &  \qquad 0 \,\,\,\, + \,\,\,\log\left(\frac{1}{t-1}\right),  \smallskip \\
	&Z & = &  \log \bigl(\frac{s}{s-1} \bigr) \,+\,\log \bigl( \frac{t}{t-1} \bigr) .
	\end{matrix}
	\end{equation}
\end{example}

\begin{remark} \label{rmk:implicitization}
	The output of Algorithm  \ref{alg:thetaSurface2} looks
	like (\ref{eq:scherkparara}) or (\ref{thetaSurfex1}). It gives the
	theta surface $\mathcal{S}$ in parametric form. Whenever the quartic is \emph{rational nodal}, meaning that
	all the components are rational and with at most nodes as singularities,
	as in Figure \ref{fig:nodalquartics},
	 the expressions
	found for $X$, $Y$ and $Z$ are $\CC$-linear combinations
	of logarithms of linear functions in $s$ and $t$. Indeed, since the singularities are nodal, 
	the differentials of \eqref{eq:omegabasis1} restrict to each rational component $\Gamma\cong \mathbb{P}^1$
	as meromorphic differentials with at most simple poles \cite[Chapter 2]{ACG}. Each such differential can be written as a 
	sum of terms of the form $\frac{1}{(t-a)}$, which integrate to $\log(t-a)$.

	From a representation as a $\mathbb{C}$-linear combination of logarithms we  find an implicit equation for
	$\mathcal{S}$ by familiar elimination techniques 
	from symbolic computation, such as {\em resultants} or {\em Gr\"obner bases}.
	Namely, we choose  constants $\alpha, \beta, \gamma \in \CC$ such that
	${\rm exp}(\alpha X)$, ${\rm exp}(\beta Y)$ and ${\rm exp}(\gamma Z)$
	are written as rational functions in $s$ and $t$, and we then eliminate
	$s$ and $t$ to obtain a trivariate polynomial $\Psi(u,v,w)$ such that
	$\mathcal{S}$ is defined locally~by
	\begin{equation}
	\label{eq:Psieqn}
	\Psi \bigl( \,{\rm exp}(\alpha X), \,{\rm exp}(\beta Y),\,{\rm exp}(\gamma Z)\, \bigr) \,\,\, = \,\,\, 0 . 
	\end{equation}
	For the output (\ref{eq:scherkparara}), we take $\alpha = \beta = i$, $\,\gamma = 1$ and
	$\,\Psi =  uv^2w-u^2v-uw+v $. With these choices, (\ref{eq:Psieqn})  
	is precisely the implicit
	equation (\ref{eq:scherk}) of Scherk's surface, times a constant.
	
	For Example~\ref{eq:amoebaofline}, the implicit equation is
	especially nice. The output (\ref{thetaSurfex1}) suggests the choices
	$\alpha=\beta=\gamma=1$ and $\,\Psi = u+v-w+1$, and hence the theta surface  is given by
	\begin{equation} \label{eq:fourterms} {\rm exp}(X) \,+\, {\rm exp}(Y) \,-\, {\rm exp}(Z) \,\,= \,\,-1 .
	\end{equation}
	An explanation for the occurrence of such exponential sums is offered
	in Section~\ref{sec5}. 
\end{remark}

We present two more  illustrations of our 
methodology for special quartics. In each case we carry out both
 Algorithm \ref{alg:thetaSurface2}
and the subsequent implicitization step as in Remark~\ref{rmk:implicitization}.

\begin{example}
Consider the quartic $q=xy(1-x^2+y^2)$. This corresponds to the pencil of conics
through $(0:i:1),(0:-i:1),(1:0:1),(-1:0:1)$. 
For the abelian integrals, we compute the partial derivatives $q_x=y(y^2-x^2+1)-2x^2y$ and $q_y=x(y^2-x^2+1)+2xy^2$. 

We first integrate the differential forms in (\ref{eq:omegabasis2}) over the line $y=0$, with local parameter $x=s$. 
On this component, $q_x=x(1-x^2)$.  The indefinite integrals are found to be
$$  \begin{matrix}
\! \int \!\! \frac{s}{-s(s^2-1)}d s = \frac{1}{2}\log (\frac{{s{+}1}}{{s{-}1}})), \quad \quad
 \int \!\! \frac{0}{-s(s^2-1)}d s = 0, \quad \quad
 \int \!\! \frac{1}{-s(s^2-1)} d s = \log(s) - \!\frac{1}{2} \log (s^2{-}1).
 \end{matrix}
$$
We next integrate over the line $x=0$, using 
(\ref{eq:omegabasis1}) with $q_y=y(y^2+1)$.
   The abelian integrals~are
$$  \begin{matrix}
\int\frac{0}{-t(1+t^2)}d t \,=\, 0 \, , \quad
 \int \frac{t}{-t(1+t^2)}d t \,=\, -\arctan (t) \, , \quad
 \int \frac{1}{-t(1+t^2)} d t \,=\, -\log(t) \,+\, \frac{1}{2} \log (1+t^2).
  \end{matrix}
$$
We conclude that the output of Step 7 in Algorithm \ref{alg:thetaSurface2} equals
\begin{align*}
&X \,\,=\,\, \frac{1}{2} \log (s+1) - \frac{1}{2} \log (s-1) , \nonumber\\
& Y \,\,=\,\,  -\arctan (t) \, = \, -\frac{1}{2i} \log (t-i) + \frac{1}{2i} \log(t+i), \\
&Z\,\,=\,\, \log(s)- \frac{1}{2} \log (s^2-1)-\log(t) + \frac{1}{2} \log (1+t^2) \nonumber. 
\end{align*}
From this we find that the implicit equation of the theta surface equals

$$
-2 \,{\rm exp}(Z)\,{\rm exp}(X)\,{\rm sin}(Y) \, - \, {\rm exp}(2X) \,- \,1 =0.
$$

\end{example}

Finding the two pairs of generating curves on a theta surface is particularly pleasant
when the underlying quartic is the union of four lines in $\PP^2$ that
are defined over~$ \QQ$.

\begin{example} \label{ex:fournicelines}
Consider the quartic $q = (y+x-1)(y-x-1)(y+x+1)(y-x+1)$.
For the first pair of generating curves,
we compute the abelian integrals  over the first two lines:
$$ \begin{matrix}
\int \!  \frac{1}{8(s - 1)}d s \,=\, \frac{1}{8}\log(s-1) ,& 
 \int \!  - \frac{1}{8s}d s \,=\, -\frac{1}{8}\log(s)  ,&
 \int \!   \frac{1}{8s(s-1)} d s \,=\, \frac{1}{8} \log (s -1)- \frac{1}{8} \log (s),  \\
\int \!  \frac{1}{8(t + 1)}d t \,=\, \frac{1}{8} \log(t + 1) ,& 
 \int \!   \frac{1}{8t}d t= \frac{1}{8} \log(t)  ,& 
 \int \!   \frac{1}{8t(t+1)} d t \,=\, - \frac{1}{8}\log(t + 1) + \frac{1}{8} \log(t).
 \end{matrix}
$$
Hence, the coordinates of the theta surface are given in terms of the parameters $s$ and $t$ by
\begin{align*}\label{thetaSurfaceEx3}
& 8X \,\,=\,\, \log(s - 1)+\log(t + 1), \nonumber\\
& 8Y\,\,=\,\,  -\log(s)\,+\,\log(t),  \\
&8Z\,\,=\,\, \log(s - 1) - \log(s)-\log(t + 1) + \log(t) \nonumber. 
\end{align*}
Choosing $\alpha=\beta=\gamma=8$ in Remark \ref{rmk:implicitization},
and hence  $\,u={\rm exp} (8X),\,v={\rm exp} (8Y)$ and $ w={\rm exp} (8Z)$, the 
implicit equation (\ref{eq:Psieqn}) of the theta surface is represented by the polynomial
$$ \Psi(u,v,w) \,\, = \,\, u^2 v w^2-2 u^2 v w-u v^2 w+u v w^2+u^2 v-4 u v w-v^2 w+u v-u w-2 v w -w .$$
The same method works fairly automatically for any arrangement of four lines in the plane.
\end{example}

\section{Degenerations of Theta Functions}
\label{sec5}

We saw in Theorem \ref{thm:riemann}
that the equation of the theta surface 
associated with a smooth plane quartic is a
Riemann theta function. However, all 
theta surfaces seen in the classical literature were computed
from quartics that are singular. In this section we explain
how  singularities induce  degenerate theta functions.
These are  finite sums of exponentials, given combinatorially by the
cells in the associated Delaunay subdivision of~$\mathbb{Z}^3$.
This furnishes a conceptual explanation of the
equations defining theta surfaces like (\ref{eq:scherk})~or~(\ref{eq:fourterms}).
Our approach is from the point of tropical geometry,
which mirrors the theory of toroidal 
degenerations~\cite{GruHu}.

In what follows we focus on the \emph{rational nodal quartic curves}.
These are quartic curves whose irreducible components are rational.
The special properties of such curves were already discussed briefly in Remark \ref{rmk:implicitization}.
Rational nodal quartics
appear in five types: an irreducible quartic with three nodes, a nodal cubic together with a line, two smooth conics, a smooth conic  with two lines, and
an arrangement of four lines.


\begin{figure}[]
\begin{center}$
\begin{array}{lllll}
\includegraphics[width=20mm]{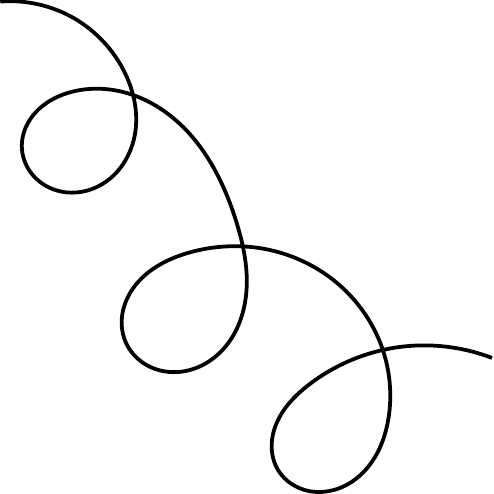}&
\includegraphics[width=20mm]{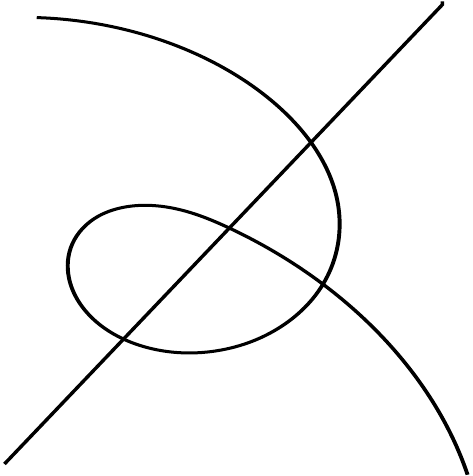}&
\includegraphics[width=20mm]{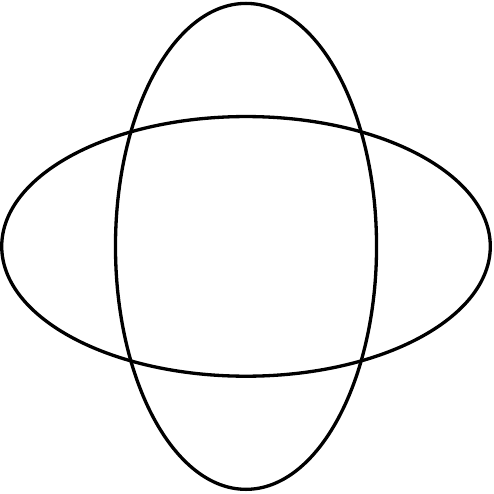}&\,\,
\includegraphics[width=12mm]{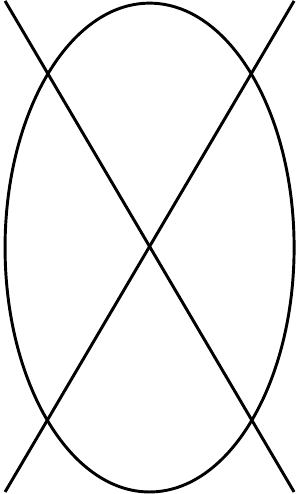}&\,\,
\includegraphics[width=20mm]{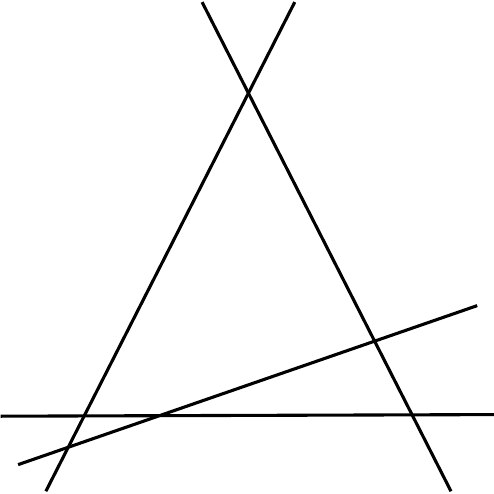}
\end{array}$
\end{center}
\caption{The five types of rational nodal quartics.
\label{fig:nodalquartics}}
\end{figure}

To each of these curves we associate its dual graph $\Gamma = (V,E)$.
This has a vertex for each irreducible component and one edge for each intersection point between two components. A node on an irreducible component counts as an intersection of that component with itself. 


\begin{figure}[h]
\begin{center}$
\begin{array}{lllll}
\includegraphics[width=20mm]{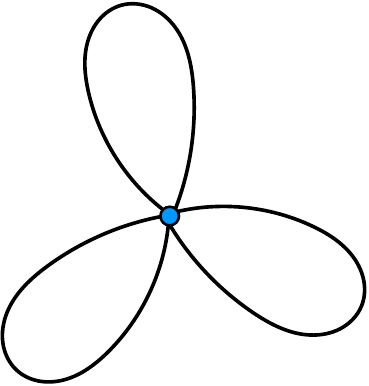}&
\includegraphics[width=20mm]{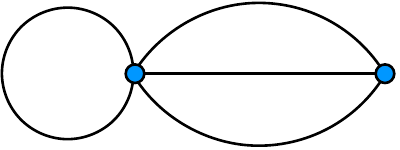}&
\includegraphics[width=20mm]{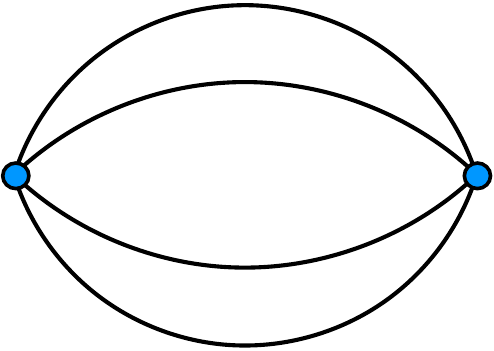}&
\includegraphics[width=20mm]{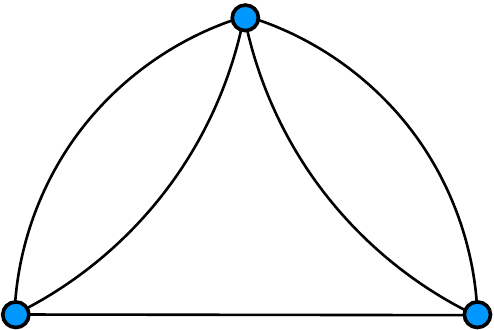}&
\includegraphics[width=20mm]{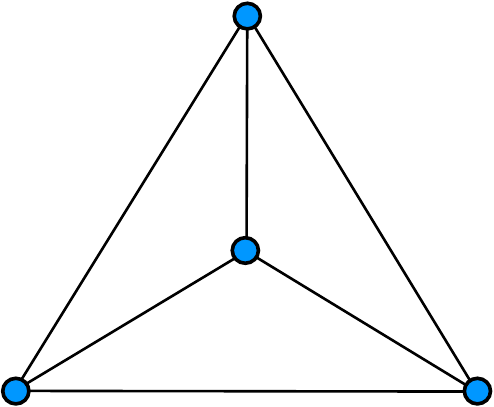}
\end{array}$
\end{center}
\caption{The dual graphs corresponding to the five types of rational nodal quartics.
\label{fig:graphs}}
\end{figure}

The dual graphs play a key role in tropical geometry,
namely in the tropicalization of curves and their Jacobians.
We follow the combinatorial construction in \cite[Section 5]{BBC}.
Fix one of the graphs $\Gamma$ in Figure~\ref{fig:graphs}
with an orientation for each edge.  The first homology group
\[ H_1(\Gamma,\mathbb{Z}) \,\,=\,\,
 \ker \left( \partial\colon \mathbb{Z}^{E} \to \mathbb{Z}^{V} \right)  \]
is free abelian of rank $3$.
Let us fix a $\mathbb{Z}$-basis  $\gamma_1,\gamma_2,\gamma_3$ for 
$H_1(\Gamma,\mathbb{Z})$. Each $\gamma_j$ is a directed
cycle in $\Gamma$.  We encode our basis in a $3 \times |E|$ matrix  $\Omega$.
The entries in the $j$-th row of $\Omega$ are the coordinates
 of $\gamma_j$ with respect to the standard basis of $\mathbb{Z}^E$.  
Our five matrices $\Omega$~are
$$ \!
\begin{bmatrix} 1 & 0 & 0 \\  0 & 1 & 0 \\ 0 & 0 & 1 \end{bmatrix}\! ,\,
\begin{bmatrix} 1 & \!\! \!-1 & 0 & 0 \\ 0 & 1 & \!\! -1 & 0 \\ 0 & 0 & 0 & 1  \end{bmatrix}\! ,  \,
\begin{bmatrix} 1 & \!\! \! -1 & 0 & 0 \\ 0 & 1 & \!\! -1 & 0 \\ 0 & 0 & 1 & \!\! -1  \end{bmatrix} \! , \,
\begin{bmatrix} 1 &  0 &  1 &  0 &  0 \\ \! -1 &  1 &  0 & \!\! -1 &  0 \\  0 & \!\!\! -1 &  0 &  0 &  \!\! -1  \end{bmatrix}\! , \,
\begin{bmatrix} 1 & \!\! \!-1 &  0 &  1 &  0 &  0 \\ \! -1 &  0 &  1 &  0 & \!\! -1 &  0 \\  0 &  1 
&\! \!\! -1 &  0 &  0 &  1  \end{bmatrix}\!. $$
We define the \emph{Riemann matrix} of  $\Gamma$ to be
the positive definite symmetric $3 \times 3$ matrix
\begin{equation}
\label{eq:periodtropical}
B \,\, := \,\,\Omega \cdot \Omega^T.
\end{equation} 
If we change the orientations and cycle bases then the matrix
$B$ transforms under the action of ${\rm GL}(3,\mathbb{Z})$ by
conjugation. The Riemann matrices of our five graphs are
the matrices in the fourth column in  Figure~\ref{fig:vallentin}.
The label \ {\tt Form} \ refers to the quadratic form
represented by~$B$.

\begin{figure}
  \includegraphics[scale=0.83]{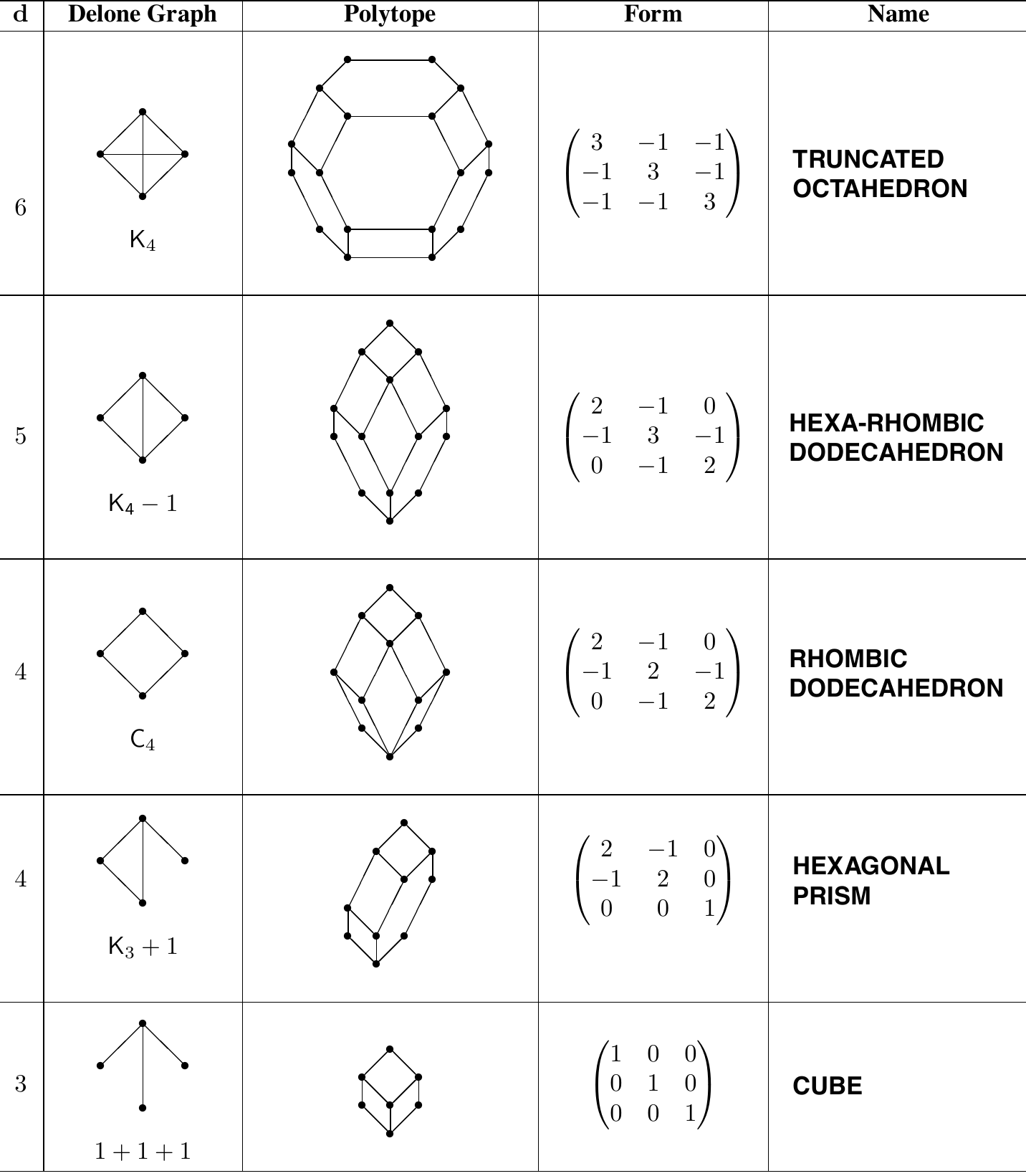}
    \caption{The Voronoi polytopes corresponding to the
  five types of rational nodal quartics, in reverse order to 
  Figure \ref{fig:graphs}.  This diagram is taken
  from the dissertation of Frank Vallentin \cite[Section 4.3, page 49]{Val}.
  The first column gives the number $d$ of edges in the dual graph.
\label{fig:vallentin}  }
\end{figure}

We now explain how a Riemann matrix $B$ of
the dual graph as in \eqref{eq:periodtropical}
induces a degeneration to a singular curve with dual graph 
as in Figure~\ref{fig:graphs}. Let $B_0$ be a fixed
(real or complex) symmetric $3 \times 3$ matrix with positive definite real part.
Consider the one-parameter family of classical Riemann matrices
\begin{equation}
\label{eq:limitfamily}
B_t \,\,:=\,\, tB+B_0 \qquad \text{for } \,\, t\geq 0.
\end{equation} 
We note that the real part of $B_t$ is always positive definite. 
If the chosen matrix $B_0$ does not belong to the hyperelliptic locus, then
the set of positive real numbers $t$ such that $B_t$ lies on the hyperelliptic locus is discrete. Thus, almost all Riemann matrices $B_t$ correspond to non-hyperelliptic curves of
genus three, and hence to smooth quartic curves in the plane.

 We consider the Riemann theta functions for these curves. More precisely, we 
 evaluate $\theta( \,\bullet\,, B_t)$ at ${\bf x}$ translated 
    by the vector $i\cdot tB\mathbf{a}$, where $\mathbf{a} = (a_1,a_2,a_3)^T \in \mathbb{R}^3$. This gives
$$ \begin{matrix}
\theta \bigl( \mathbf{x}-i\cdot tB\mathbf{a},B_t \bigr) & = & \sum_{n\in \mathbb{Z}^3} \mathbf{e}\left(-\frac{1}{2}n^T(tB+B_0) n + i\cdot n^T\left(\mathbf{x}-i\cdot tB\mathbf{a}\right) \right) \medskip \\
& =  & \sum_{n\in \mathbb{Z}^3}\mathbf{e} \left( - \frac{1}{2}\left(n^T Bn-2n^TB\mathbf{a} \right)t - \frac{1}{2}n^TB_0 n + i\cdot n^T\mathbf{x}\right) \medskip \\
& = & \sum_{n\in \mathbb{Z}^3} \mathbf{e}\left(-\frac{1}{2}(n^T Bn-2n^TB\mathbf{a})t \right)\cdot \mathbf{e}\left( - \frac{1}{2}n^TB_0n+i\cdot n^T\mathbf{x}  \right) \label{eq:thetalimit}.
\end{matrix}
$$
As $t\to +\infty$,
the term $\mathbf{e}\left(-\frac{1}{2}(n^T Bn-2n^TB\mathbf{a})t \right)$ converges
 if and only if  $\,n^TBn-2n^TB\mathbf{a}\geq 0$.
For each $n \in \mathbb{Z}^3$, this condition is a linear inequality in ${\bf a}$, which can be 
rewritten as follows:
\begin{equation}
\label{eq:voronoicell}
\mathbf{a}^T B \mathbf{a} \,\,\,\leq \,\,(\mathbf{a}-n)^TB(\mathbf{a}-n).
\end{equation}
Hence, in order for the theta function above to converge to a degenerate theta
function, we  must choose  $\mathbf{a}$ in such a way  that \eqref{eq:voronoicell} is satisfied for every $n\in \mathbb{Z}^3$. The positive definite quadratic form given by $B$
defines a metric on $\mathbb{R}^3$. The condition \eqref{eq:voronoicell} 
says that, among all lattice points $n$ in $\mathbb{Z}^3$,
 the origin is a closest one to $\mathbf{a}$ in this metric.
 This means that   $\mathbf{a}$ is contained in the {\em Voronoi cell} 
 with respect to the lattice $\mathbb{Z}^3$ and the metric defined by~$B$. 
 
 The Voronoi cell is a $3$-dimensional polytope. It belongs to the
 class of {\em unimodular zonotopes}.  
The third column of Figure  \ref{fig:vallentin}
 shows this for each of the five types of nodal
 quartics in Figure \ref{fig:nodalquartics}. The translates of the Voronoi cell 
 by vectors in $\mathbb{Z}^3$
 define a tiling of $\mathbb{R}^3$. Their boundaries
 form an infinite $2$-dimensional polyhedral complex.
 See \cite[Figure 15]{BBC} for the case of four lines, on the right
 in Figures \ref{fig:nodalquartics} and \ref{fig:graphs}.
 This surface is the {\em tropical theta divisor} in $\mathbb{R}^3$.
 It can be viewed as a polyhedral model for our
degenerate theta surface. For details on these objects 
see \cite{BBC,Chan} and references therein.
We encourage our readers to~spot Figure \ref{fig:graphs} in \cite[Figures 1 and 8]{Chan} and to
examine the tropical Torelli map in \cite[Theorem 6.2]{Chan}.

We now assume that $\mathbf{a}$ is a vertex of the Voronoi cell.
We write  $\mathcal{D}_{{\bf a},B}$ for the set of all vectors
$n \in \mathbb{Z}^3$ for which equality holds in (\ref{eq:voronoicell}).
This set is finite for each Riemann matrix $B$ derived from a graph in
Figure~\ref{fig:graphs}. We note the following behavior for the summands above:
\begin{equation}
\label{eq:summandsabove}
\,\mathbf{e}\biggl(-\frac{1}{2}(n^T Bn-2n^TB\mathbf{a})t \biggr) 
\,\,\, \xrightarrow{t \rightarrow +\infty} \,\,\,
\begin{cases}\, \,1 & {\rm if} \,\, n \in \mathcal{D}_{{\bf a},B},
\\ \,\,0 & {\rm if}\,\, n \in \mathbb{Z}^3 \backslash \mathcal{D}_{{\bf a},B}. \end{cases} 
\end{equation}
From this we shall infer the following theorem, which is our main result in this section.

\begin{theorem}\label{thm:degtheta}
Fix a vertex ${\bf a}$ of the Voronoi cell for the 
degeneration (\ref{eq:limitfamily}). The associated theta function 
is the following finite sum over all vertices of the Delaunay polytope dual to~${\bf a}$:
\begin{equation}
\label{eq:finitethetasum}
  \sum_{n\in \mathcal{D}_{{\bf a},B}} 
   \mathbf{e}\left( - \,\frac{1}{2}n^TB_0n\,+\,i\cdot n^T\mathbf{x}  \right) .
\end{equation}
The number of summands in (\ref{eq:finitethetasum}) equals $\,8$ for a rational quartic,
 $6$ for a nodal cubic plus line, $4$ or $6$ for two conics, 
$4$ or $5$ for a conic plus two lines,  and $4$ for four lines.
This recovers the equations for theta surfaces given by Eiesland in
\cite[eqns.(5),(6)]{Eies08} and \cite[eqn.(5)]{Eies09}.
\end{theorem}

 \begin{proof}
The subdivision of $\mathbb{R}^3$ dual to the Voronoi decomposition
is the Delaunay subdivision; see \cite[Section 5]{BBC}  or \cite[Section 4.2]{Chan}.
Its cells are dual to those of the Voronoi decomposition. In particular, each
vertex ${\bf a}$ of the Voronoi polytope corresponds to a $3$-dimensional 
{\em Delaunay polytope}.  
The vertices of the Delaunay polytope are the elements of the set $\mathcal{D}_{{\bf a},B}$.

For instance, consider the case of four lines, which is
listed last in Figures \ref{fig:nodalquartics}, \ref{fig:graphs} and
 first in Figure \ref{fig:vallentin}. The Voronoi cell is the
{\em permutohedron}, also known as the truncated octahedron.
Each of its $24$ vertices ${\bf a}$ is dual to a tetrahedron
in the Delaunay subdivision. See \cite[Example 5.5]{BBC}
and the left diagram in \cite[Figure 15]{BBC}.
Hence all Delaunay polytopes are tetrahedra, i.e.~$|\mathcal{D}_{{\bf a},B}| = 4$.
This explains the four terms in 
(\ref{eq:fourterms}) or~\cite[eqns.~(6)]{Eies08}.

The four other Delaunay subdivisions are obtained by
moving the matrix $B$ to a lower-dimensional stratum in the
tropical moduli space, proceeding downwards in \cite[Figure 8]{Chan}.
The resulting Delaunay polytopes are obtained by fusing the tetrahedra in
 \cite[Figure 15]{BBC}. Hence all  Delaunay polytopes 
 ${\rm conv}(\mathcal{D}_{{\bf a},B})$
  are naturally triangulated into unit tetrahedra.
 
 We now present a list, up to symmetry, of all vertices ${\bf a}$ 
 of the Voronoi polytopes. Our five special Riemann matrices $B$ here appear in
 the order given in  Figure \ref{fig:vallentin}. In each case, we display   the 
 Delaunay set  $\mathcal{D}_{{\bf a},B}$. This is the support
 of the degenerate theta function  in (\ref{eq:finitethetasum}):
$$
 \begin{matrix}
 \text{\em type of curve}  & \text{\# \em orbit} & {\bf a}^T & \mathcal{D}_{{\bf a},B} \medskip \\
  \text{$4$ lines} & 24 &
\bigl(\frac{3}{4},\frac{1}{2},\frac{1}{4} \bigr) & \{ (000),(100),(110),(111) \} \medskip \\
   \text{conic + $2$ lines} & 8 &
\bigl(\frac{3}{4},\frac{1}{2},\frac{1}{4} \bigr) & \{ (000),(100),(110),(111) \} \smallskip \\
  \text{conic + $2$ lines} & 10 &
\bigl(\frac{5}{8},\frac{1}{4},\frac{5}{8}\bigr) & \{ (000),(001),(100),(101), (111)\} \medskip \\
  \text{$2$ conics} & 8 &
\bigl(\frac{3}{4},\frac{1}{2},\frac{1}{4} \bigr) & \{ (000),(100),(110),(111) \} \smallskip \\
  \text{$2$ conics} & 6 &
\bigl(\frac{1}{2},1,\frac{1}{2}\bigr) & \{ 
(000),(010),(011),(110),(111),(121) \} \medskip \\
 \text{cubic+line} & 12 & 
\bigl(\frac{2}{3},\frac{1}{3},\frac{1}{2} \bigr) & 
\{ (000),(001),(100),(101),(110),(111)\}
\medskip \\
 \text{rational quartic} & 8 & 
 \bigl(\frac{1}{2},\frac{1}{2},\frac{1}{2} \bigr) & 
\hspace{-10mm} \{ (000),(001),(010),(011),
\medskip \\
 &  & 
 & \quad
\hspace{9mm} (100),(101),(110), (111)
 \} \medskip \\
\end{matrix}
 $$
The column {\em \# orbit} gives the cardinality of
the symmetry class of the vertex ${\bf a}$ of the Voronoi cell.
We see that the Delaunay polytope
$\,{\rm conv}(\mathcal{D}_{{\bf a},B})\,$  is either a
 tetrahedron, an Egyptian pyramid, an octahedron, or a cube.
 Given the type in Figure~\ref{fig:nodalquartics}, 
for a suitable choice of Riemann matrix $B$ 
  and Voronoi vertex ${\bf a}$, we recover precisely the
 tetrahedron in \cite[eqn.~(6)]{Eies08},
 the octahedron in \cite[eqn.~(6)]{Eies08},
 and the cube in  \cite[eqn.~(5)]{Eies09}.
 Eiesland's coefficients $A,B,C,\ldots$ for these theta surfaces
 are determined by the fixed symmetric matrix $B_0$
 that define the degeneration~(\ref{eq:limitfamily}).
In conclusion, equation (\ref{eq:summandsabove})  implies that 
$\,\theta \bigl( \mathbf{x}-i\cdot tB\mathbf{a},B_t \bigr) \,$
converges to the function given by the 
finite sum in (\ref{eq:finitethetasum}),
with $\mathcal{D}_{{\bf a},B}$ as derived above.
 \end{proof}

In the table above, the same four-element set $\mathcal{D}_{{\bf a},B}$
occurs for 4 lines, for conic + 2 lines and for 2 conics. This is explained by
Remark \ref{rmk:onsameconic} because all three choices
are possible for the basis of a fixed pencil of conics.
Our running example belongs to this tetrahedron case.

\begin{example}[Scherk's minimal surface] \label{ex:vierzwei}
We here use theta functions to recover the surface in Figure~\ref{polytope}.
The rational nodal quartic (\ref{eq:quartic1}) consists of a smooth conic and two lines. 
The second row in Figure \ref{fig:vallentin} shows that the corresponding tropical period matrix is
	$ \,B= \begin{tiny} \begin{pmatrix} \,\,2 & \! -1 & \,0 \,\\  -1 & 3 & \!\! -1 \\ \,\, 0 \,& \!\! -1 & 2 \,\end{pmatrix}\end{tiny} $.
	
We consider the degenerate theta function of Theorem \ref{thm:degtheta} with the data
$$ \begin{small}
	 B_0 \,=\, i\cdot \begin{pmatrix} -1 & 0 & 0 \\ 0 & 1 & 0 \\ 0 & 0 &-1   \end{pmatrix},
	  \quad \mathbf{a} \,=\, \begin{pmatrix} 
	 3/4 \\ 1/2 \\ 1/4 \end{pmatrix}, \quad \mathbf{x} 
	= \frac{1}{2\pi}\begin{pmatrix} 2X \\ -X+Y-iZ \\ -2Y \end{pmatrix} .
	\end{small} $$
	Then	 $\,\mathcal{D}_{\mathbf{a},B} \,=\, \{ (0,0,0),\, (1,0,0),\, (1,1,0),\, (1,1,1) \}$,
	and the four-term sum in (\ref{eq:finitethetasum}) equals
\begin{equation}
\label{eq:scherkExp}
	\begin{matrix} & 1 - {\rm exp}(2iX)+ {\rm exp}(iX+iY+Z) - {\rm exp}(iX-iY+Z) \\ = &
- 2i \cdot {\rm exp}(2i X) \cdot \bigl( \,
\sin(X) \,-\, \sin(Y)\exp(Z)\, \bigr).
\end{matrix}
\end{equation}
 In the parentheses on the right we see
expression (\ref{eq:scherk}) from the beginning of this paper.
\end{example}

 \begin{example}[Irreducible quartic]
This class of theta surfaces was studied  by Eiesland in \cite{Eies09}.
His ``unicursal quartic'' is the rational quartic with three nodes
on the left in Figure \ref{fig:nodalquartics}. Here,
	 $B$ is the identity matrix and the Voronoi cell is the cube
	 with vertices $\left(\pm\frac{1}{2},\pm\frac{1}{2},\pm\frac{1}{2}\right)$.
	  Fixing the vertex $\mathbf{a}^T = \left(\frac{1}{2},\frac{1}{2},\frac{1}{2}\right)$,
	  the Delaunay polytope is the cube with vertex set $\{0,1\}^3$.
	  
We choose an arbitrary real symmetric $3 \times 3$ matrix $B_0$, and we abbreviate
\begin{equation}
\label{eq:Ann}	\begin{matrix} A_{n} \,\,: =\,\, \mathbf{e} \bigl(-\frac{1}{2}n^t B_0 n \bigr)
	\qquad {\rm for} \,\, \,n \in \{0,1\}^3.  \end{matrix}
\end{equation}	
Note that $A_{(000)} = 1$.
Writing $\mathbf{x}= -i \cdot(X,Y,Z)$,  the degenerate theta function in 
(\ref{eq:finitethetasum}) equals
$$ \begin{matrix}
	\lim_{t\to +\infty}\theta\left( \mathbf{x}-i\cdot tB\mathbf{a},B_t \right) 
	 \,\,=\,\,\, 1
	\,+\,A_{(100)}\,\mathbf{e}( X)
	\,+\,A_{(010)}\,\mathbf{e}( Y) 
	\,+\,A_{(001)}\,\mathbf{e}( Z) \qquad \qquad \\ \qquad \qquad
	+\,A_{(011)}\,\mathbf{e}( Y{+}Z) 
	+A_{(101)}\,\mathbf{e}( X{+}Z)
	+A_{(110)}\,\mathbf{e}( X{+}Y)
	+A_{(111)}\,\mathbf{e}( X{+}Y{+}Z).
\end{matrix}
$$
This is precisely the theta surface derived
in the theorem in \cite[page 176]{Eies09}. Here we have
	\begin{equation*}
	A_{(100)}A_{(010)}A_{(001)}A_{(111)} \,\,=\,\, A_{(000)} A_{(011)}A_{(101)}A_{(110)}.
	\end{equation*}
	This follows directly from (\ref{eq:Ann}), and it matches 
	Eiesland's identity in	\cite[eqn.~(6)]{Eies09}.
\end{example}

\section{Algebraic Theta Surfaces}
\label{sec6}

 Theta surfaces are usually transcendental. But, in some special cases, it can happen that~they are algebraic.
 These cases were classified by Eiesland \cite{Eies08}. Our aim is to present his result.
 
 \begin{example} \label{ex:TSdegree4}
 We begin by showing that the following quartic surface is a theta surface:
 \begin{equation} \label{eq:TSdegree4}
 Y^4\,-\,4XY^2\,-\,4X^2\,+\,8Z\,\,=\,\,0.
 \end{equation}
 The underlying quartic curve consists of a cuspidal cubic together with its cuspidal tangent:
 $$q\,=\,(y^2-x^3)y.$$
 We evaluate the abelian integrals over the line $y=0$ and the cubic, 
 parametrized respectively  $x=s$ and $x=t$. It turns out that all
 required  antiderivatives are algebraic functions, namely
$$ \begin{matrix}
X & = & \frac{1}{2}\int\frac{s}{s^3}d s \,-\,\int\frac{t}{t^3}d t & = &  - \frac{1}{2s} +\frac{1}{t} ,\medskip \\
Y & = & \frac{1}{2}\int\frac{s^{3/2}}{s^3}d s \,-\,\int\frac{0}{t^3}d t & = &  - \frac{1}{\sqrt{s}} \, ,  \medskip \\
Z & = & \frac{1}{2}\int\frac{1}{s^3}ds \,-\,\int\frac{1}{t^3}d t  & = &  -\frac{1}{4s^2}+\frac{1}{2t^2}.
\end{matrix} $$
This is a parametrization of the rational surface (\ref{eq:TSdegree4}), which is
singular along a line at~infinity.
\end{example}

\begin{example}\label{Quadric}
Another one is the quadric whose equation $3XY-Z=0$. It arises from the quartic
which is the union of the four concurrent lines $x=-y$, $x=y$, $2x=-y$, $2x=y$. 
%
\end{example}

Eiesland \cite{Eies08} identified all scenarios where the abelian integrals are essentially rational.

\begin{theorem}[Eiesland] \label{thm:eiesland}
Every algebraic theta surface is rational and has degree $2, 3, 4, 5$~or~$6$.
The underlying quartic has rational components and none of its singularities are nodes.
\end{theorem}

In Example \ref{ex:TSdegree4} we already saw our first algebraic theta surface. In the next 
examples we present further surfaces, of degrees $5, 4, 6, 3$, in this order. In each case, 
the quartic curve satisfies the condition in the second sentence of Theorem~\ref{thm:eiesland}.
The section will conclude with a discussion on degenerations of abelian functions. 
This will establish the link  to Section~\ref{sec5}.

\begin{figure}[!htb]
      \begin{minipage}{0.50\textwidth}
     \centering
     \includegraphics[width=0.95\linewidth]{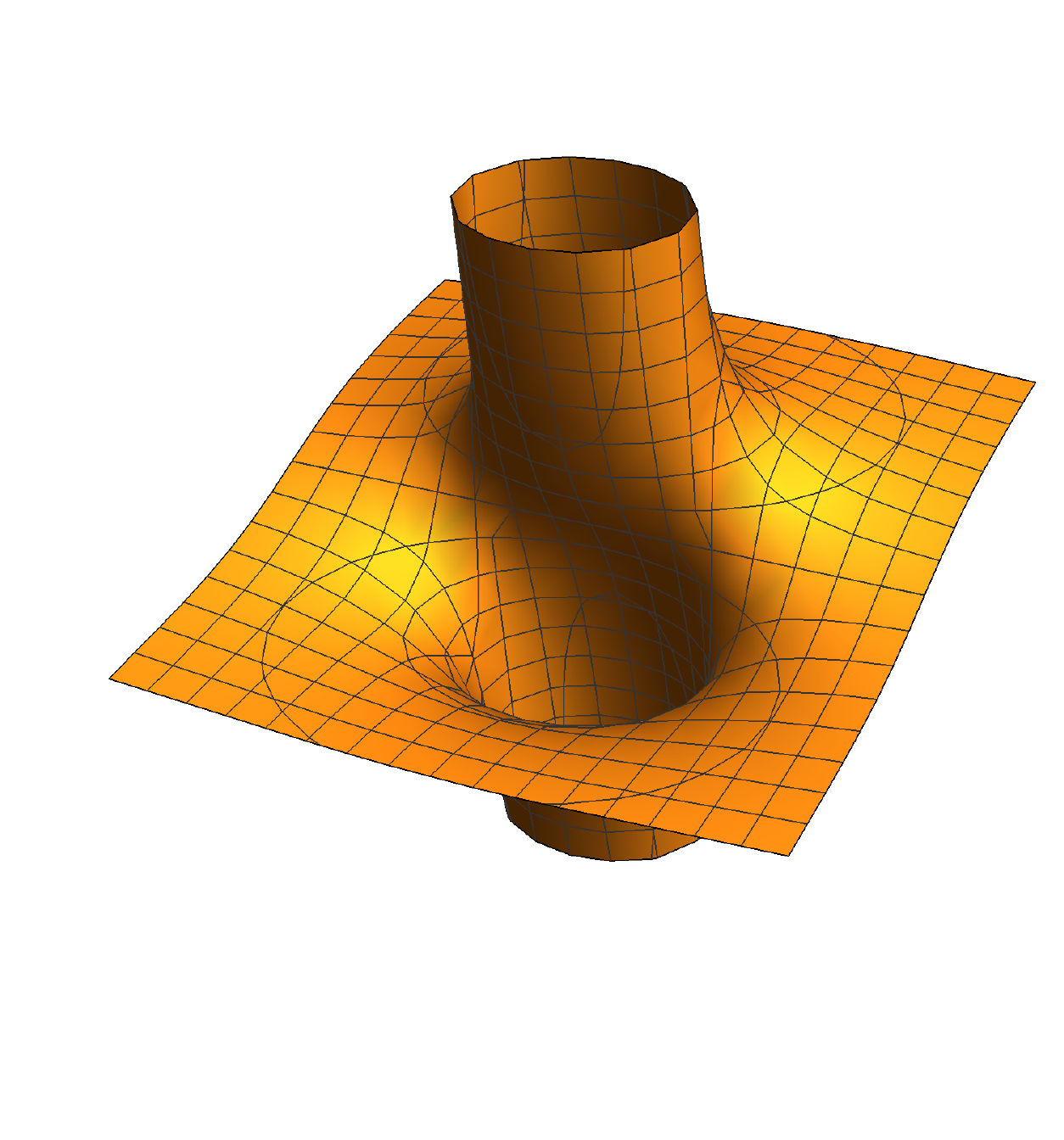}\\
   \end{minipage} \hfill
      \begin{minipage}{0.50\textwidth}
     \centering
     \includegraphics[width=.6\linewidth]{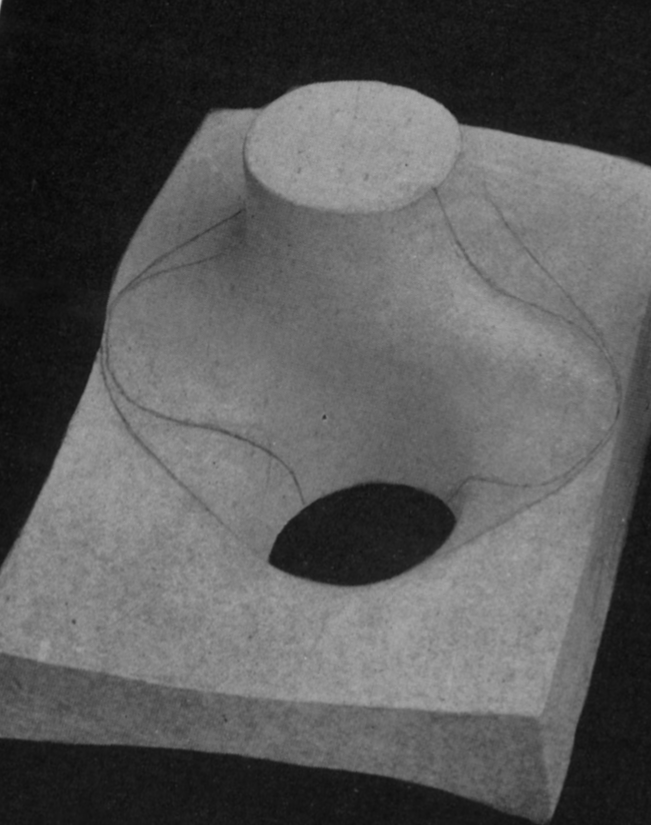}\\
   \end{minipage} 
   \vspace{-0.35in}
   \caption{The cardioid surface in Example \ref{ex:cardiodsurface}.
      On the right is Eiesland's plaster model~\cite{Eies08}.
   \label{fig:CardioidSurface}}
 \end{figure}
 
\begin{example}\label{ex:curvelittle}
The following toric curve a is rational quartic with a triple point:
$$ Q \,\,=\,\, x^4 - yz^3 .$$
Following \cite[Example 4.3]{Little83},
the abelian integrals on this curve evaluate to
$$ \begin{matrix}
X &= &\int sds +\int{t}d t & = & \frac{1}{2}(s^2+t^2), \smallskip \\
Y & = & \int s^4d s +\int t^4d t & = & \frac{1}{5}(s^5+t^5),\smallskip \\
Z & = & \int d s +\int d t & = & s+t.
\end{matrix} $$
Elimination of the parameters $s$ and $t$ reveals the equation for this quintic theta surface:
\begin{equation} \label{eq:TSdegree5}
Z^5-20X^2Z+20Y\,\,= \,\,0.
\end{equation}
Little \cite{Little83} generalizes this and other theta surfaces to higher dimensions.
He presents a geometric derivation of theta divisors in $\CC^g$ from 
(degenerate) canonical curves of genus $g$.
\end{example}

Our next example shows how transcendental theta surfaces can degenerate to algebraic surfaces.
This kind of analysis plays an important role in Eiesland's proof of Theorem \ref{thm:eiesland}.

 \begin{example} \label{ex:TSdegree3}
 We consider an irreducible rational quartic with two singular points, namely 
  one ordinary cusp and one tacnode. The following curve
  has these properties for general $\lambda$:
  $$q\,\,=\,\,x^3y+\lambda x^4+(1-\lambda) x^2y -y^2 \qquad {\rm with} \quad
  q_y\,\,=\,\,x^3 + (1-\lambda)x^2 -2y . $$
Following Eiesland's derivation in \cite[Section III]{Eies08}, we choose
the rational parametrization
$$ \begin{matrix}
x\,\,=\,\,\frac{w^2\,+\,(\lambda+1)w}{w\,+\,1}
\quad {\rm and} \quad 
y\,\,= \,\,\frac{(w^2\,+\,(\lambda+1)w)^2}{w\,+\,1}
\end{matrix}
$$
The differential form $dx$ can be written in terms of the curve parameter $w$ as follows:
$$ \begin{matrix} dx \,\, = \,\, \bigl( 1 + \frac{\lambda}{(w+1)^2} \bigr) dw .
\end{matrix} $$
We substitute these expressions into  \eqref{eq:omegabasis1}, 
and we find the three antiderivatives
$$ \begin{matrix} 
\frac{\log(w+\lambda+1)}{\lambda+1} - \frac{\log(w)}{\lambda + 1} \,  , \quad  -w \, , \quad
\frac{(\lambda - 1)\log(w+\lambda+1 )}{(\lambda+1)^3}-\frac{(\lambda-1)\log(w )}{(\lambda+1)^3} -
 \frac{(\lambda - 1)w - \lambda-1}{(\lambda+1)^2w^2+(\lambda+1)^3w}. \end{matrix}
 $$
 The two generating curves are found by setting  $w=s$ and $w=t$. Hence our theta surface is
$$ \begin{matrix}
X & = & \frac{\log(s+\lambda+1)}{\lambda+1} - \frac{\log (s)}{\lambda + 1}
\,\,+\,\,\frac{\log(t+\lambda+1)}{\lambda+1} - \frac{\log(t)}{\lambda + 1} , \smallskip \\
 Y & = &  -\,s\,- t,\qquad \qquad \smallskip \\
Z & = & \frac{(\lambda - 1)\log(s+\lambda+1 )}{(\lambda+1)^3}-\frac{(\lambda-1)\log(s )}{(\lambda+1)^3} - 
\frac{(\lambda - 1)s - \lambda-1}{(\lambda+1)^2s^2+(\lambda+1)^3s} \smallskip \\   & 
&\quad \quad +\frac{(\lambda - 1)\log(t+\lambda+1 )}{(\lambda+1)^3}-\frac{(\lambda-1)\log(t )}{(\lambda+1)^3} - \frac{(\lambda - 1)t - \lambda-1}{(\lambda+1)^2t^2+(\lambda+1)^3t}.
\end{matrix} $$
From this parametrization we infer that the theta surface is transcendental for $\lambda \not= -1$. 

Now let $\lambda=-1$. Then $q=x^3y-x^4+2yx^2-y^2$
and the tacnode is now a {\em node-cusp}.
This is a singular point obtained by merging a node and a cusp.
Using  the calculus identity $ \,
{\rm lim}_{\lambda \rightarrow -1}  \bigl\{\frac{\log(s+\lambda+1)}{\lambda+1} 
- \frac{\log(s)}{\lambda + 1} \bigr\} = 
\frac{d}{ds} {\rm log}(s) =
\frac{1}{s}$, the parametrization of our surface becomes
$$ 
\begin{matrix} X\,=\,\frac{1}{s }+\frac{1}{t },\quad \quad Y\,=\,-s-t, \quad \quad
Z\,=\,\frac{3s+2}{6s^3}+\frac{3t+2}{6t^3}. \end{matrix}
$$
Eliminating $s$ and $t$, we obtain the implicit equation.
$$ 2X^3Y+3X^2Y+6X^2-6YZ+6X\,\,=\,\,0. $$
Hence, the theta surface is a rational quartic for $\lambda=-1$,
and it is transcendental for $\lambda \not=-1$.
\end{example}

The maximum degree of any algebraic theta surface is six. The next example  attains~this.

\begin{example} \label{ex:ats6}
 Following Eiesland \cite[p.~381--383 and VI on p.~386]{Eies08},
we consider the quartic
$$ q \,\,= \,\, (y-x^2)^2 \,+\, 2xy(y-x^2) \,+\, y^3 . $$
The unique singular point, at the origin, is a  {\em tacnode cusp}. The theta surface is given by
$$ \begin{matrix}
X & = &  \int \! \frac{s+1}{s^4} ds \,+\, \int \! \frac{t+1}{t^4} dt & = &
-\frac{1}{2s^2} - \frac{1}{3s^3} \,- \,\frac{1}{2t^2} - \frac{1}{3t^3} ,
\smallskip \\
Y & = &  \int \! \frac{1}{s^2}ds \, +\,\int \! \frac{1}{t^2}dt , & = &
-\frac{1}{s}\,-\,\frac{1}{t},
\smallskip \\
Z & = &  \int \! \frac{2s+1}{s^6} ds \,+\,\int \! \frac{2t+1}{t^6} & = &
-\frac{1}{5s^5} - \frac{1}{2s^4} \,-\,\frac{1}{5t^5} - \frac{1}{2t^4}.
\\
\end{matrix}
$$
The implicit equation is found to be
$$
4 Y^6-24 Y^5-60 X Y^3+45 Y^4+180 X Y^2-180 X^2+180 Y Z-180 Z \,\,\,=\,\,\, 0.
$$
This sextic looks different from that in \cite[VI on p.~386]{Eies08}
because of a coordinate change.
\end{example}

We conclude our  panorama of algebraic theta surfaces with two classical cubic surfaces. 

\begin{example}[Cardioid Surface] \label{ex:cardiodsurface}
We first consider the cardioid $q=(x^2+y^2-2x)^2-4(x^2+y^2)$. Note that $q_y=4x^2y + 4y^3 -8xy -8y$. We choose a rational parametrization as follows:
$$ \begin{matrix}
x\,\,=\,\,\frac{4(1-w^2)}{(w^2+1)^2}\,,
\quad \hbox{hence} \quad
dx=\frac{8w(w^2 - 3)}{(w^2+1)^3}dw \,,
\quad \hbox{and} \quad
y\,\,= \,\,\frac{8w}{(w^2+1)^2}.
\end{matrix}
$$
We substitute this into the differential forms in (\ref{eq:omegabasis1}), and 
we compute the three antiderivatives:
$$ 
\begin{matrix}
 \int \frac{w^2-1}{2(1+w^2)^2}dw =  -\frac{w}{2(w^2 + 1)} \, , \quad
\int \frac{-w}{(1+w^2)^2}dw =  \frac{1}{2(w^2 + 1)} \, ,  \quad
  \int \frac{-1}{8}dw =  -\frac{w}{8}.
\end{matrix}
$$
Similar to the computations in the previous examples, we obtain the theta surface as follows:
\begin{equation*}
\label{eq:cardiodsurface}
\begin{matrix} X\,=\,-\frac{s}{2(s^2 + 1)}-\frac{t}{2(t^2 + 1)},\quad \quad Y\,=\,\frac{1}{2(s^2 + 1)}+\frac{1}{2(t^2 + 1)}, \quad \quad Z\,=\,-\frac{s}{8}-\frac{t}{8},  \medskip \\ 8X^2Z + 8Y^2Z - 4YZ - X \,\,=\,\,0.
\end{matrix}
\end{equation*}
Two pictures of the cardioid surface, from  the 19th and 21st century,
are shown in Figure~\ref{fig:CardioidSurface}.
\end{example}

\begin{figure}[h!]
\vspace{-0.08in}
   \begin{minipage}{0.50\textwidth}
     \centering
     \includegraphics[width=.9\linewidth]{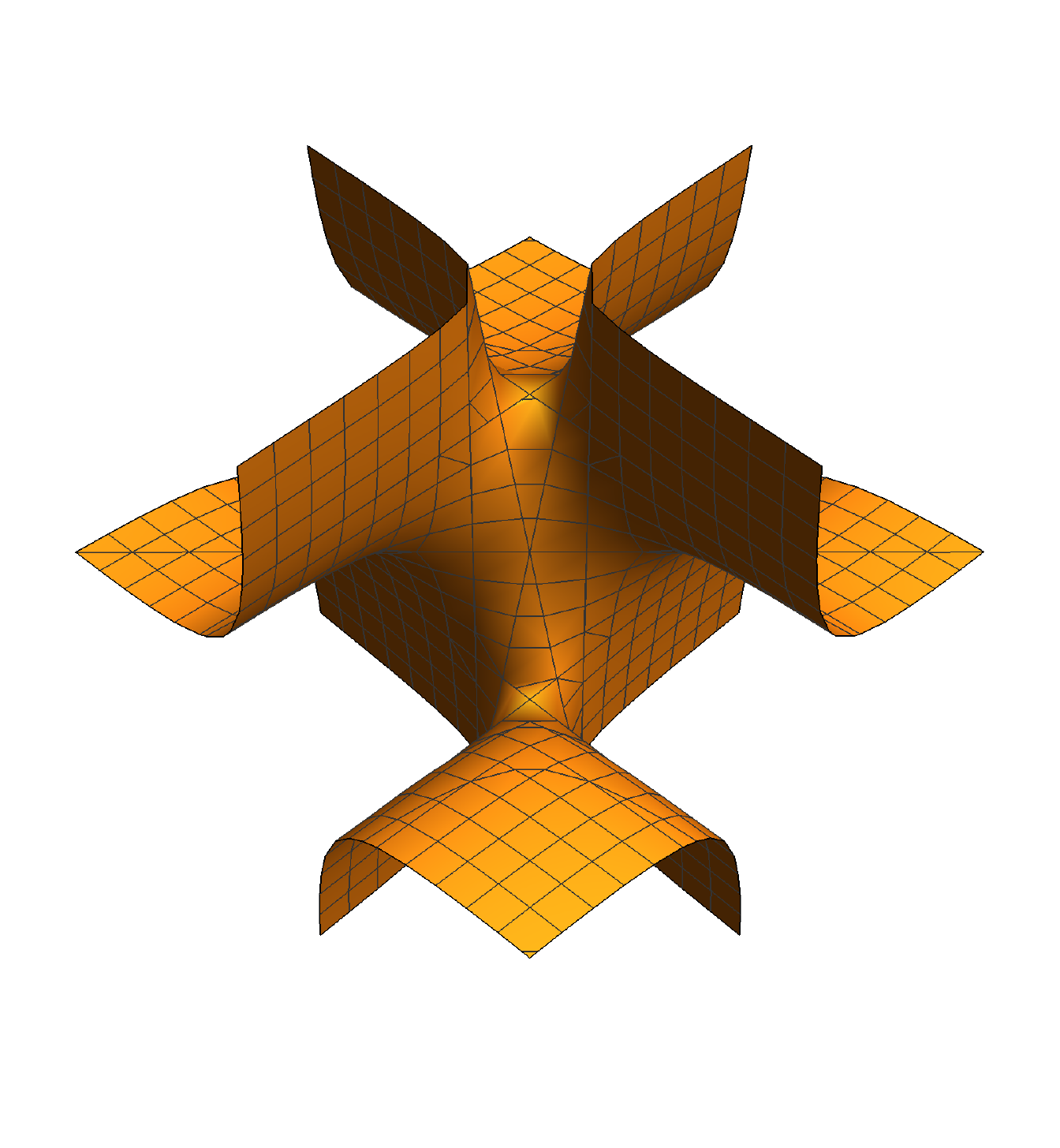}\\
   \end{minipage}\hfill
   \begin{minipage}{0.50\textwidth}
     \centering
     \includegraphics[width=.6\linewidth]{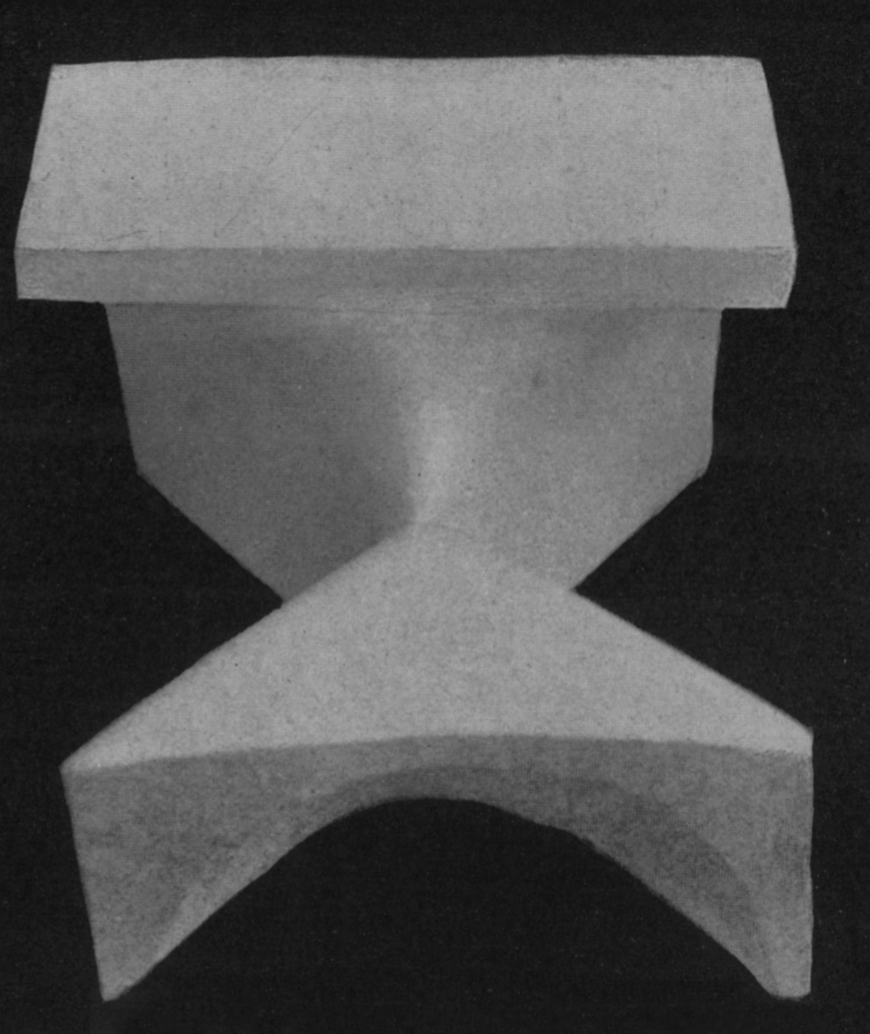}\\
   \end{minipage}
   \vspace{-0.3in}
   \caption{The Deltoid Surface.    On the right ist   Eiesland's plaster model~\cite{Eies08}.
   \label{fig:DeltoidSurface}}
\vspace{-0.04in}   
\end{figure}

\begin{example}[Deltoid Surface] \label{ex:deltoidsurface}
We consider the deltoid curve with parametrization
$$ \begin{matrix}
x\,\,=\,\,\frac{4}{(w+1)^2},
\quad {\rm and} \quad 
y\,\,= \,\,\frac{4}{w^2+6w+9}.
\end{matrix}
$$
This defines the quartic $\,q=y^2+x^2-2xy+x^2y^2-2x^2y-2xy^2$.
The abelian integrals give
\begin{equation*}
\begin{matrix}
&X & = & \int \frac{1}{(s+1)^2}ds+\int \frac{1}{(t+1)^2}dt& = & -\frac{1}{(s + 1)}-\frac{1}{(t + 1)}, 
\smallskip \\
& Y & = & \int \frac{1}{(3+s)^2}ds+\int \frac{1}{(3+t)^2}dt  & = &- \frac{1}{s+3}-\frac{1}{(t+3)}, 
\smallskip \\
&Z & = & \int \frac{1}{4}ds+\int \frac{1}{4}dt  &= &\frac{s}{4}+\frac{t}{4}.
\end{matrix}
\end{equation*}
Elimination yields the following cubic equation. This surface is exhibited in
Figure~\ref{fig:DeltoidSurface}:
\begin{equation*}
4XYZ + 4XY + 2XZ - 2YZ + 3X - Y\,\,=\,\,0.
\end{equation*}
This equation can be transformed into the cubic $\,XYZ = a_1 X+a_2 Y+a_3 Z\,$ given by Chern \cite[p.~2]{Chern} 
via a linear change of coordinates, similar to that in  \cite[page 376, equation~(11)]{Eies08}.
We saw the algebraic theta surfaces of lowest degree three 
in Figures \ref{fig:CardioidSurface} and \ref{fig:DeltoidSurface},
in views that take us back to the 19th century.
 We will learn more about the plaster models
in Section~\ref{sec7}. 
\end{example}

In Section \ref{sec5} we derived special transcendental theta surfaces
by degenerations from Riemann's theta function. An example
was the formula for Scherk's surface in (\ref{eq:scherkExp}).
This raises the question whether our algebraic theta surfaces
can be obtained in a similar manner.
While the answer is affirmative, the details are complicated
and we can only offer a glimpse. The role of the theta function is now played
by a variant called the {\em sigma function} \cite[\S 5]{naka}.

We illustrate this for the singular quartic of Example \ref{ex:curvelittle},  given by the affine equation
\begin{equation}
\label{eq:xzsingquar}
 x^4 - z^3 \,\,=\,\, 0. \qquad  \qquad
 \end{equation}
This belongs to the family of {\em $(3,4)$-curves} \cite[eqn (5.1)]{BEL}.
These plane curves are defined~by
\begin{equation}
\label{eq:34curve}
\qquad
x^4 - z^3 \,+\, \lambda_1 x^2z + \lambda_2 x^2 
+ \lambda_3 xz +\lambda_4 x + \lambda_5 z + \lambda_6 \,\,=\,\, 0 .  
\end{equation}
The binomial (\ref{eq:xzsingquar})
is the most degenerate instance where all six coefficients $\lambda_i$ are zero. 
Such curves, and the more general $(n,s)$ curves, were introduced in the
theory of integrable systems by Buchstaber, Enolski and Leykin \cite{BEL}
and studied further by Nakayashiki \cite{naka}. They 
considered the sigma function associated to (\ref{eq:34curve}).
This is an
abelian function which generalizes Klein's classical sigma function for hyperelliptic curves. 
The sigma function for a $(3,4)$-curve is a multigraded power series in $X,Y,Z$
whose coefficients are polynomials in $\lambda_1,\lambda_2,\ldots,\lambda_6$.
By \cite[Example 4.5]{BEL},
the term of lowest degree in the sigma function equals
$$\sigma_{3,4} \,\,\,=\,\,\,\,Z^5\, -\, 5X^2Z \,+\, 4 Y . $$
This is precisely the quintic in (\ref{eq:TSdegree5}),
after the coordinate scaling $(X,Y,Z)  \mapsto (2X,5Y,Z)$.
Thus, our algebraic theta surface arises from the sigma function 
by setting $\lambda_1 = \cdots = \lambda_6 = 0$.

Similarly, the theta surface in
Example \ref{ex:ats6} is closely related to the sextic $\sigma_{2,7}$ in
 \cite[Example 4.5]{BEL}.
 The polynomials $\sigma_{n,s}$ are known as {\em Schur-Weierstrass polynomials}.
 These play a fundamental role for rational analogs of abelian functions, and hence in the
 design of special solutions to the KP equation.  For details we refer to
 \cite{BEL,naka} and the references therein. Even the 
 genus $3$ case offers  opportunities for further research. It would be interesting to
 revisit this topic from the perspectives of
    theta surfaces and tropical geometry, as in Section~\ref{sec5}.

\section{A Numerical Approach for Smooth Quartics}
\label{sec4}

In this section we assume that the given quartic curve $\mathcal{Q}$  is nonsingular  in $\PP^2$. 
Then $\mathcal{Q}$ is a compact Riemann surface of genus $3$.
Riemann's Theorem \ref{thm:riemann} shows that its theta divisor coincides
with its theta surface, up to an affine transformation. 
We here validate that result computationally
using current tools from numerical algebraic geometry.
 We sample points on
  the theta surface using Algorithm~\ref{alg:thetaSurface1} below, and we then check that Riemann's theta function vanishes at these points using the Julia package  \texttt{Theta.jl} 
  by Agostini and Chua~\cite{Julia}.

\medskip

 \begin{algorithm}[H]\label{alg:thetaSurface1}
    \KwIn{ The inhomogeneous equation $q(x,y)$ of a smooth plane quartic}
    \KwOut{A point on the corresponding theta surface in $\RR^3$ or $\CC^3$}
         \KwSty{Step 1:} Specify two points $p_1$ and $p_2$ on the quartic. \\
          \KwSty{Step 2:} Take two other points $p_1'$ and $p_2'$ nearby $p_1$ and $p_2$ respectively. \\
    \KwSty{Step 3:} Compute the following triples of integrals numerically:
     \begin{equation*} \begin{matrix}
c_1 & = &   (\int_{p_1}^{p_1'}\omega_1 , \int_{p_1}^{p_1'}\omega_2 , \int_{p_1}^{p_1'}\omega_3),  &
c_2 & = &
      (\int_{p_2}^{p_2'}\omega_1 , \int_{p_2}^{p_2'}\omega_2 , \int_{p_2}^{p_2'}\omega_3).
     \end{matrix}
     \end{equation*}
  \\
    \KwSty{Step 4:} Output the sum $c_1+c_2$.  
    \caption{Sampling from a theta surface given its plane quartic }
\end{algorithm} 

\medskip

This algorithm is similar to Algorithm \ref{alg:thetaSurface2}.
However, the difference is that computation is now done by numerical evaluation.
Indeed, when the polynomial $q(x,y)$ defines a smooth quartic, it is impractical to work with
 an algebraic formula for $y$ in terms of $x$, so we employ numerical methods even for  
 Steps 1 and 2 above.   Of course, when such an expression is available, it can be used 
 to strengthen the numerical computations, as we will see later.

The central point of Algorithm \ref{alg:thetaSurface1} is computing the abelian integrals in Step 3.
Such integrals appear throughout mathematics, from algebraic geometry to number theory and integrable systems, and there is extensive work in evaluating them numerically. Notable 
implementations are the library \texttt{abelfunctions} in SageMath~\cite{SwiDec}, 
the package \texttt{algcurves} in Maple~\cite{DecHoe}, and the MATLAB code presented in \cite{MATLAB}. 
The software we used for our experiments is the package \verb|RiemannSurfaces| in SageMath
due to Bruin, Sijsling and Zotine~\cite{BruSijZot}.

The underlying algorithm views a plane algebraic curve $\mathcal{Q}$ as a ramified 
cover $\mathcal{Q}\to\mathbb{P}^1$~of the Riemann sphere $\mathbb{P}^1$ via the projection $(x,y)\mapsto x$. The package lifts paths from $\mathbb{P}^1$ to paths on the Riemann surface $\mathcal{Q}$ and integrates the abelian differentials in \eqref{eq:omegabasis1} along these paths via certified homotopy continuation. 
In order to carry this out, it is essential to avoid the ramification points of the projection to $\mathbb{P}^1$.
This is done by computing the Voronoi decomposition of the Riemann sphere
$\mathbb{P}^1$ given by the branch points 
of $\mathcal{Q}\to \mathbb{P}^1$. The integration paths are obtained from edges of the
Voronoi cells. Avoidance of the ramification points is also 
a feature in the other packages such as \texttt{algcurves} 
and \texttt{abelfunctions}.

It is important to note that avoiding the ramification points conflicts with our
desire to create real theta surfaces and to
work with Riemann matrices and theta equations over~$\RR$.
The cycle basis that is desirable for revealing the real structure, as in \cite{Sil},
forces us to compute integrals near  ramification points. 
We had to tweak the method in  \cite{BruSijZot} to make this~work.

\smallskip

In what follows we present a case study
that illustrates Algorithm~\ref{alg:thetaSurface1}.
Our instance is the {\em Trott curve}.
 This is a smooth plane quartic $\mathcal{Q}$, defined by the inhomogeneous polynomial
 \begin{equation}
 \label{eq:trottpolynomial}
 q(x,y) \,\,=\,\, 144(x^4+y^4)-225(x^2+y^2)+350x^2y^2+81 . 
 \end{equation}
The Trott curve is a widely known example of a real quartic whose $28$ bitangent lines are all real 
and touch at real points.
 It is also a \emph{M-curve}, meaning that the real locus $\mathcal{Q}_{\mathbb{R}}$ has four connected components, which is the highest possible number for a real genus 3 curve.
 
\begin{figure}[h]
	\centering
	\includegraphics[width=0.4\linewidth]{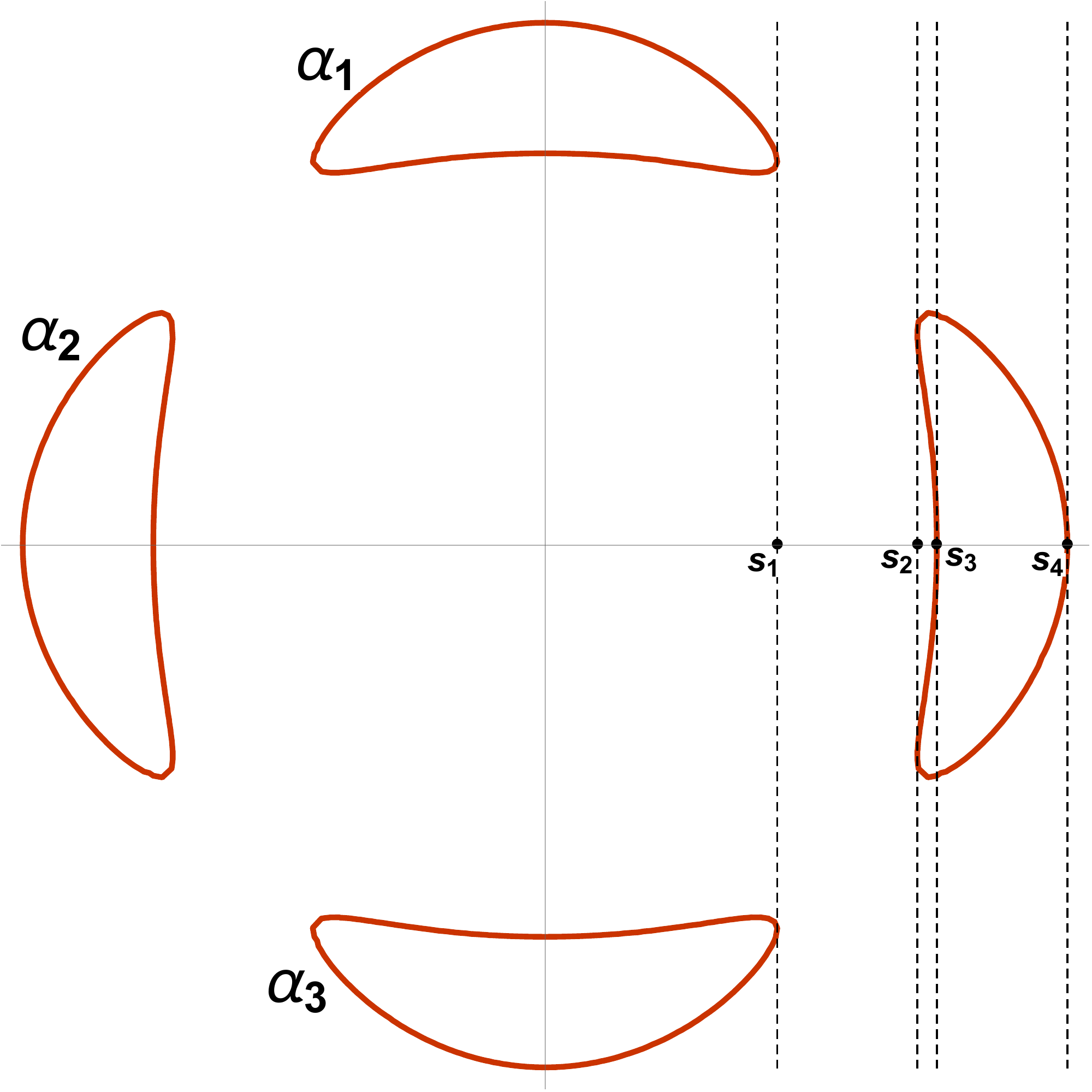} \\
\caption{The Trott curve, together with some branching points and some homology cycles.}
  \label{fig:trott}
\end{figure}

Given any M-curve $\mathcal{Q}$, there exists a symplectic basis 
for the homology group $H^1(\mathcal{Q},\mathbb{Z})$ 
whose associated period matrix 
$\Pi$ in (\ref{periodMatrix}) respects the real structure (cf.~\cite{Sil}). With such a choice of basis, the
Riemann matrix $B$ in (\ref{eq:riemannmatrix}) is real, and the theta function
defines a surface in real  $3$-space $\mathbb{R}^3$.
According to Theorem \ref{thm:riemann}, this surface is precisely
  our theta surface. 
 
 We now verify this numerically. The first step consists in identifying a real period matrix for $\mathcal{Q}$.
 To do so, we follow Silhol \cite{Sil}. First we observe that the Trott curve is highly symmetric: 
  its automorphism group is the dihedral group $D_4$, generated by the automorphisms
\begin{equation}\label{eq:autTrott} 
 (x,y) \mapsto (x,-y) ,\qquad (x,y)\mapsto(-x,y) ,\qquad  (x,y)\mapsto (-x,-y), \qquad  (x,y)\mapsto (y,x) .
\end{equation}
We choose a symplectic basis of $H^1(\mathcal{Q},\mathbb{Z})$ as follows: the cycles 
$\alpha_1,\alpha_2,\alpha_3$ are as indicated in Figure \ref{fig:trott}, where we take 
$\alpha_1,\alpha_3$ with clockwise orientation and $\alpha_2$ with 
counterclockwise orientation. Furthermore, referring again to Figure \ref{fig:trott}, we take $\beta_1,\beta_3$ to be the two cycles lying over the interval $[s_1,s_2]$ and intersecting $\alpha_1,\alpha_3$ respectively. Instead, the path $\beta_2$ is the one lying over the interval $(-\infty,-s_4]\cup[s_4,+\infty)$. Then, with these choices, 
there exist real numbers $a_j$ and purely imaginary numbers $b_i$ such that
the two $3 \times 3$ blocks in \eqref{periodMatrix} satisfy
\begin{equation}\label{eq:periodtrott}
\Pi_{\alpha} \,=\, \begin{pmatrix} 0 & -a_1 & 0 \,\\ -a_1 & 0 & a_1 \, \\ \,a_2 & -a_2 & a_2 \,\end{pmatrix}, \qquad
 \Pi_{\beta} \,= \, \begin{pmatrix}\, b_1 & 2b_1 &\, b_1 \\ b_1 & 0 & -b_1 \\ \, b_2 & 0 & b_2 \end{pmatrix}.
\end{equation}
The scalars $a_j$ are real because the paths $\alpha_1,\alpha_2,\alpha_3$ 
and the differentials $\omega_1,\omega_2,\omega_3$ are real.
The scalars  $b_i$ are purely imaginary because, by construction, the paths 
$\beta_1,\beta_2, \beta_3$ are anti-invariant with respect to complex conjugation 
on $\mathcal{Q}$. The symmetries in the matrices $\Pi_\alpha$ and $\Pi_\beta$ reflect the action
 of the automorphism group of \eqref{eq:autTrott}. Since $\Pi_{\alpha}$ is real and $\Pi_{\beta}$ is purely imaginary, 
we conclude that the Riemann matrix $\,B=-i\cdot \Pi_{\alpha}^{-1}\cdot \Pi_\beta\,$ has 
all its entries real.

At this point it should be straightforward to compute all the above explicitly. However, as we see from Figure \ref{fig:trott}, the paths that we have chosen pass through the ramification points of the projection $\mathcal{Q}\to \mathbb{P}^1$ to the $x$-axis. Hence, we cannot use the existing routines, such as \verb|RiemannSurfaces| or \verb|abelfunctions| straight away. To solve this problem we mix the numerical strategy together with the symbolic one in Algorithm \ref{alg:thetaSurface2}.
Indeed, since the Trott curve is highly symmetric, we can compute its points symbolically in terms of radicals: more precisely, for any $x\in\mathbb{C}$, the four corresponding points $(x,y)$ on the Trott curve are given by
\begin{equation}\label{eq:trottsolution} 
y \,\,=\,\,  \pm {\frac{\sqrt{225-350x^2\pm \sqrt{39556 x^4 - 27900 x^2 + 3969}}}{12\sqrt{2}}} .
\end{equation}
With this formula, we represent the abelian integrals on the Trott curve as symbolic integrals in the single variable $x$, 
and we then compute these
numerically. We have implemented this strategy in Maple
 and we found the parameters in the period matrices of \eqref{eq:periodtrott} to be
\begin{align*}
a_1&\,=\,-0.02498252478, & a_2&\,=\,0.03154914935,\\
b_1&\,=\,\phantom{-} 0.01384015941i, & b_2&\,=\,0.02348847438i.
\end{align*} 
Consequently, the Riemann matrix for the Trott with our choice of homology basis equals
\begin{equation*}
B\,\,=\,\,
\begin{pmatrix}0.926246 &  0.553994 &  0.372252 \\
0.553994 &   1.10799 &  0.553994 \\
0.372252 & 0.553994 & 0.926246
\end{pmatrix}.
\end{equation*}
The next step is to compute some points in the theta surface $\mathcal{S}$. This can be done via Algorithm \ref{alg:thetaSurface1} using the methods of \verb|RiemannSurfaces|, as long as we integrate away from the ramification points of the projection to the $x$ axis. Otherwise, we can employ again the symbolic representation of \eqref{eq:trottsolution}, together with numerical integration. 

We proceed as follows: first, we can choose the points $p_1,p_2$ in Algorithm \ref{alg:thetaSurface1} in such a way that the line passing through them is bitangent to $\mathcal{Q}$ and parallel to the $x$ axis. Then we compute the integrals in Algorithm \ref{alg:thetaSurface1} for points in small neighborhoods of $p_1,p_2$. 
Further details  on such computations can be found in
  Section \ref{sec7}, following the definition of an envelope in~\eqref{eq:stmaptangent}.
The rows below are five points on $\mathcal{S}$ we obtained with Maple:
\[
10^{-5} \cdot 
\begin{pmatrix}   
2.58293  & \:  -4.55191  & \: -6.38839 \\
2.53203  & \: -4.46200   & \: -6.26220 \\
2.48111  & \: -4.37204   & \: -6.13596 \\
2.43015  & \: -4.28204   & \: -6.00964 \\
2.37916  & \: -4.19200   & \: -5.88327
\end{pmatrix}. 
\]


The last step is to check that these points are zeroes of the theta function $\theta(\mathbf{x},B)$, up to an affine transformations. To do so, we employ the Julia package \texttt{Theta.jl}  that is described in \cite{Julia}. This 
package is the latest software for computing with theta functions.
 It is especially optimized for the case of small genus and, 
 more importantly, for repeatedly evaluating a theta function $\theta(\mathbf{x},B)$ at multiple points $\mathbf{x}$ 
 for the same fixed Riemann matrix~$B$. This allows for a fast evaluation of the theta function which is very helpful for our problems.

In our situation, we now have a sample of points $\mathbf{x}_1,\dots,\mathbf{x}_N$ on the theta surface.
These are given numerically.
We consider the transformed points $\Pi^{-1}_{\alpha}\mathbf{x}_1,\dots,\Pi^{-1}_{\alpha}\mathbf{x}_N$.
According to the proof of Riemann's Theorem \ref{thm:riemann}, there exists a
vector $\kappa\in \mathbb{C}^3$ such that the translated theta function 
$\theta(\textbf{x}+\kappa,B)$ vanishes on 
$\Pi^{-1}_{\alpha}\mathbf{x}_1,\dots,\Pi^{-1}_{\alpha}\mathbf{x}_N$. 
According to the full version of Riemann's Theorem \cite[Appendix to \S 3]{Mum1},
which incorporates theta characteristics,
the vector $\kappa$ can be assumed to have the form 
$\kappa = \frac{1}{2}\left( iB\varepsilon + \delta\right)$, where 
$\varepsilon,\delta \in \{0,1\}^3$. In particular, there are only $64$ 
possible choices for $\kappa$. We can check explicitly all of the $64$ possibilities.

In our experiments, we computed $N=10404$ points on the surface $\mathcal{S}$, and we evaluated
\[  m(\varepsilon,\delta) := \max_{i=1,\dots,N} |\theta(\mathbf{x}_i+\kappa,B)| \]
for each of the 64 possible choices of $(\varepsilon,\delta)$. This was computed by \texttt{Theta.jl}
on a standard laptop in approximately 9.6 minutes. We found that 
\[ m\left( \varepsilon_0, \delta_0 \right) \approx 6\cdot 10^{-12}, \qquad \text{ for } \qquad \varepsilon_0 = \begin{pmatrix} 1 \\ 1 \\ 1 \end{pmatrix},\,\, \delta_0 = \begin{pmatrix} 0 \\ 0 \\ 1 \end{pmatrix}. \]
For all the other choices of the pair $\varepsilon,\delta$, we determined that $10^{-3} \leq m(\varepsilon,\delta) \leq 2$. 

This computation amount to a numerical verification of Riemann's Theorem  \ref{thm:riemann}.
We have
\[ \mathcal{S} = \Pi_{\alpha}\cdot (\Theta_B - \kappa_0), \qquad \text{ for }\,\, \kappa_0 = \frac{1}{2}\left( iB\varepsilon_0 + \delta_0\right). \]
To conclude, this gives also a real analytic equation for $\mathcal{S}$. Indeed, for any $\kappa=\frac{1}{2}\left( iB\varepsilon + \delta\right)$,
 the translated theta divisor $\Theta_B-\kappa$ is cut out by the theta function with characteristic
\begin{align*} 
\theta[\varepsilon,\delta](\mathbf{x},B)  & = \sum_{n\in \mathbb{Z}^3}\mathbf{e}\left( -\frac{1}{2}\left(n+\frac{\varepsilon}{2}\right)^t B \left( n + \frac{\varepsilon}{2} \right) + i\left( n+\frac{\varepsilon}{2} \right)^t\left( \mathbf{x}+\frac{\delta}{2} \right) \right) \\
& = \sum_{n\in \mathbb{Z}^3}\exp \left(-\pi \left(n+\frac{\varepsilon}{2}\right)^t B \left( n + \frac{\varepsilon}{2} \right)  \right)\cdot \cos \left( 2\pi \left( n+\frac{\varepsilon}{2} \right)^t\left( \mathbf{x}+\frac{\delta}{2} \right) \right) .
\end{align*} 
This is a real analytic function since the matrix $B$ is real.

\section{Sophus Lie in Leipzig}
\label{sec7}

Felix Klein held the professorship for geometry at the University of Leipzig until  1886 
when he moved to G\"ottingen. In the same year,  Sophus Lie was appointed
to be Klein's successor and he moved from
Christiania (Oslo) to Leipzig. Lie also became one of the three directors of the Mathematical Seminar, 
an institution that Felix Klein had founded with the aim of strengthening the connection between education and research. In his first years at Leipzig,
 Lie was busy  with completing his major work {\em Theory of Transformation Groups}
 with the assistance of  Friedrich Engel. It was released  in three volumes in 1888, 1890 and 1893.
  Thereafter, the subject of double translation surfaces moved back in the focus of his teaching and research,
  and it caught   the attention of the mathematical community for the first time.

In what follows we discuss notable historical developments, we revisit
Lie's pre-Leipzig work on these surfaces, and we show
how it relates to our discussion in the previous sections.
In 1892 Lie published an article explaining how theta surfaces can be
parametrized by abelian integrals \cite[p.~481]{LieGes2}.
He invited two of his Leipzig students, Richard Kummer and Georg Wiegner, to rework the classification 
he had given in 1882  by means of abelian integrals.
The work of Kummer and Wiegner was published in their doctoral theses \cite{Kum,Wie}. Under the supervision of Lie's assistant Georg Scheffers, the two students also constructed a series of twelve plaster models 
that visualize the diverse shapes exhibited by theta surfaces. It is surprising that the models were 
commissioned  by Lie, who, unlike his predecessor Klein, had 
not been known for an engagement in popularizing mathematics in this manner.

The collection of mathematical models at the University of Leipzig was initiated by Felix Klein in 1880. 
At the end of the 19th century, the collection included around $350$ models and drawings. During the 
20th century, many models were lost or broken. A project for cataloging and restoring the collection 
was initiated by Silvia Sch\"oneburg in 2014. A catalogue describing all
$240$ remaining models is expected to be published in 2021.
There are some very rare models in the collection, among them 
nine of the surfaces created by  Lie's students. These plaster models and their mathematics are the topic
of the third author's diploma thesis~\cite{Struwe}, submitted to Leipzig University in 2020.
It was her find of the models by Kummer and Wiegner that brought us together for our project on theta surfaces.

In 1895 Poincar{\'e} presented his proof of the relation between double translation surfaces and Abel's Theorem \cite{Poin1895,Poin95}. This led to the idea that these surfaces can
be seen as theta divisors of Jacobians \cite[p.2]{Chern}. 
Darboux \cite{Darb} and Scheffers \cite{Sch} also published variants of the proof.
Eiesland \cite{Eies08,Eies09}  completed the classification initiated by Kummer and Wiegner. He also constructed plaster models for some of his surfaces (cf.~Figures \ref{fig:CardioidSurface}
and \ref{fig:DeltoidSurface}). Eiesland's plaster models of theta surfaces were donated to 
the collection at John Hopkins University.

Later on, our theme found its way into modern research. 
Shiing-Shen Chern \cite{Chern} 
characterized theta surfaces in terms of the \emph{web geometry} 
that was developed by Blaschke and Bol in the 1930's. 
In 1983, John Little \cite{Little83}  studied the theory for curves of genus $g \geq 3$, and he proposed a solution of the \emph{Schottky problem} of recognizing Jacobians in terms of {\em translation manifolds}. 
More precisely, he showed that a principally polarized abelian variety
of dimension $g$ is 
the Jacobian of a non-hyperelliptic curve if and only if its theta divisor can be 
written locally as a Minkowski sum of $g-1$ analytic curves.
Little's article \cite{Little92}  connects this point of view to the integrable systems approach
(cf.~\cite{DFS})  to the Schottky problem.  

\smallskip

The doctoral theses of Kummer and Wiegner built on Lie's earlier results.
One of these is the recovery of generating curves for a theta surface $\mathcal{S}$ from the equation of 
$\mathcal{S}$ by means of differential geometry. This was  helpful for 
constructing real surfaces in situations when the abelian integrals delivered complex values.
We state Lie's result using a slight modification of the set-up in Section \ref{sec2}.
Fix a quartic $\mathcal{Q}\subset \mathbb{P}^2$  and let $\mathcal{L}(0)$ be a line that is \emph{tangent} to $\mathcal{Q}$ at a smooth point $p_1(0)$ and intersects $\mathcal{Q}$ in other two points $p_3(0),p_4(0)$. Using the points $p_1(z)$ near $p_1(0)$, where
$z$ is a local coordinate, we obtain a parametrization as in \eqref{eq:stmap}:
\begin{equation}\label{eq:stmaptangent}
(s,t)\,\,\mapsto \,\,\,\Omega_1(p_1(s)) \,+\,\Omega_1(p_1(t)).
\end{equation}
The surface $\mathcal{S}$ is  the Minkowski sum $\mathcal{S}=\mathcal{C}+\mathcal{C}$, 
where $\mathcal{C}$ is the curve $z\mapsto \Omega_1(p_1(z))$. 
The scaled curve  $2\cdot \mathcal{C}$ lies in $\mathcal{S}$. This curve
was called an \emph{envelope} by Lie.

In classical differential geometry, an {\em  asymptotic curve} on a surface $\mathcal{S}$ is
 a curve whose tangent direction at each point has normal curvature zero on $\mathcal{S}$.
 This means that the tangent direction at each point is isotropic with respect to the 
 second fundamental form of~$\mathcal{S}$.

\begin{theorem}[Lie] \label{thm:lieasymptotic}
Let $\mathcal{S} = \mathcal{C}+\mathcal{C}$ be the theta surface  given
 by the parametrization~(\ref{eq:stmaptangent}). Then the envelope 
$\,2\cdot \mathcal{C}\,$ is an asymptotic curve of the surface $\mathcal{S}$. 
\end{theorem}

 This result appears in 
\cite[p.~211]{LieGes2}, albeit in a different formulation
that emphasizes minimal surfaces.
Our version  in Theorem \ref{thm:lieasymptotic} was presented by Kummer in \cite[p.~15]{Kum}.

\smallskip
 
   Lie's reconstruction is remarkable in that it solves  \emph{Torelli's problem} for genus $3$ curves.
   Indeed, Lie found his result several decades before  Torelli \cite{Tor}  
proved his famous theorem in algebraic geometry.  Torelli's problem asks to recover an algebraic curve 
 $\mathcal{C}$ from its Jacobian $J(\mathcal{C})$ together with the theta divisor 
 $\Theta \subset J(\mathcal{C})$.
   In our situation, once we recover the envelope $\mathcal{C}$ as the asymptotic curve of
   the surface $\mathcal{S}$, we can reconstruct the quartic curve $\mathcal{Q}$ 
   as in Remark \ref{rmk:quarticfromcurves}.  
   Note that      this reconstruction technique also works
    for singular quartics. 
   
   Algebraic geometers will notice a connection between Lie's approach and Andreotti's  
geometric  proof \cite{Andr} of Torelli's theorem. Indeed, the second fundamental 
form of $\mathcal{S}$ is the differential of the Gauss map
$\, \mathcal{S} \rightarrow {(\mathbb{P}^2)}^*$.
This map associates to each point of $\mathcal{S}$ its tangent space in $\mathbb{C}^3$.  
Andreotti  observed that the Gauss map of the theta divisor 
 $\Theta \subset J(\mathcal{C})$ is branched 
precisely over the curve in ${(\mathbb{P}^2)}^*$
dual to $\mathcal{C} \subset \mathbb{P}^2$.
Hence $\mathcal{C}$ can be recovered thanks to the biduality theorem.
It would be interesting to further study
  Lie's differential-geometric approach to the Torelli problem via
the   Gauss map. One natural question is whether 
Theorem~\ref{thm:lieasymptotic} extends to curves of higher genus 
and how this relates to Andreotti's method.

\begin{example}
	To illustrate Lie's result, we determine an envelope for  Scherk's minimal surface
	directly from the equation \eqref{eq:scherk}.
	After computing the second fundamental form, we see that a curve $(X(t),Y(t),Z(t))$ in
	the surface is asymptotic if and only if it satisfies
	\[ \frac{\dot{X}(t)^2}{\sin^2(X(t))}\, =\, \frac{\dot{Y}(t)^2}{\sin^2(Y(t))}\, , \quad \text{ or equivalently } \quad  \frac{\dot{X}(t)}{\sin(X(t))} \,=\, \pm\frac{\dot{Y}(t)}{\sin(Y(t))}. \]
	The solutions to this	 differential equation are given by the following two families of curves:
	\[ \frac{\tan\left( X/2 \right)}{\tan\left( Y/2 \right)} \,=\, c
	\qquad {\rm and} \qquad \tan\left( \frac{X}{2} \right)\cdot \tan\left( \frac{Y}{2} \right) \,=\, c 
	\qquad {\rm for} \quad c\in \mathbb{R}\setminus \{0\}. \] 
Consider a curve from the first family. Setting $X=2\arctan(t)$, we can parametrize it as
	\[ \left( 2\arctan(t)\,,\,\, 2\arctan\left( \frac{t}{c}\right)\,,\,\,
	\log \left( \frac{c^2+t^2}{c(t^2+1)} \right) \right). \]
	The last expression comes from the fact that $Z=\log\left( \frac{\sin(X)}{\sin(Y)} \right)$ holds
	on Scherk's surface.
	Now, setting $c=\frac{1}{5}$ and using Theorem \ref{thm:lieasymptotic}, we obtain the generating curve
	$\mathcal{C}$ given by
$$	\left( \arctan(t), \arctan\left( 5t\right),\frac{1}{2}\log \left( \frac{1+(5t)^2}{5(t^2+1)} \right) \right). $$
The resulting representation $\mathcal{S} = \mathcal{C} + \mathcal{C}$ is
precisely the one we presented in equation~\eqref{eq:scherkpara2}.
\end{example}


Already in 1869, Lie studied the parametrization of {\em tetrahedral theta surfaces}
\begin{equation}
\label{eq:TTS} \mathcal{S} \,\,= \,\, \bigl\{\, \alpha \cdot \exp(X)
\,+\,\beta \cdot \exp(Y)\,+ \, \gamma \cdot \exp(Z)\,=\,\delta \,\bigr\}. 
\end{equation}
Here $\alpha,\beta,\gamma,\delta$ are nonzero constants. These surfaces play a prominent role
in Theorem \ref{thm:degtheta}. The adjective ``tetrahedral'' refers to the fact that the
Delaunay polytope is a tetrahedron.
  Example \ref{ex:vierzwei} shows
that Scherk's surface is tetrahedral, after a coordinate change over~$\CC$.

 Lie proved that tetrahedral theta surfaces admit infinitely many representations
 $\mathcal{S} = \mathcal{C}_1 + \mathcal{C}_2$.
This was already mentioned in Remark  \ref{rmk:onsameconic}.
We present Lie's method for identifying these infinitely many pairs of generating curves.
 A key tool is the \emph{logarithmic transformation}
\begin{equation*}
\label{eq:logT}
X=\log(U),\quad Y=\log(V),\quad Z=\log(W) .
\end{equation*}
This  transforms the surface $\mathcal{S}\subset \mathbb{C}^3$ into the 
plane $\mathcal{P}\subset (\mathbb{C}^*)^3$ defined by the equation
\begin{equation} \label{eq:log(TTS)}
\mathcal{P} \,\,=\,\, \bigl\{\, \alpha \cdot U\,+\,\beta \cdot V\,\,+\,\,\gamma \cdot W\, = \, \delta \, \bigr\}.
\end{equation}
The generating curves in $\mathcal{S}=\mathcal{C}_1+\mathcal{C}_2$ 
correspond to curves $\mathcal{D}_1,\mathcal{D}_2$ such that $\mathcal{P} = \mathcal{D}_1\cdot \mathcal{D}_2$. Here  $\mathcal{D}_1\cdot \mathcal{D}_2$ denotes the {\em Hadamard product} of the two curves,
i.e.~the set   obtained from the coordinatewise product of all points in $\mathcal{D}_1$ with all points in $\mathcal{D}_2$. Lie studied this alternative formulation and found infinitely many pair of \emph{lines} $\mathcal{D}_1,\mathcal{D}_2\subset (\mathbb{C}^*)^3$ such that $\mathcal{P}=\mathcal{D}_1\cdot \mathcal{D}_2$.  

We shall state Lie's result more precisely. 
 The action of the group of translations on $\mathbb{C}^3$ corresponds under the logarithmic transformation to the action of the torus $(\mathbb{C}^*)^3$ on itself.
Thus, we are free to rescale the coordinates $U,V,W$.
  In particular,  we can assume that our plane $\mathcal{P}$ and the desired
lines  $\mathcal{D}_1$ and $\mathcal{D}_2$ contain  the point $\mathbf{1}=(1,1,1)$.
With this, the identity
      $\mathcal{P}=\mathcal{D}_1\cdot \mathcal{D}_2$ implies  $\mathcal{D}_1,\mathcal{D}_2 \subset \mathcal{P}$. On the theta surface side, this corresponds to translating the surface and the curves until all of them pass through the origin $\mathbf{0}=(0,0,0)$.

We next consider the closure 
 of the plane $\mathcal{P}$ and the lines 
$\mathcal{D}_1,\mathcal{D}_2$ inside the projective space $\mathbb{P}^3$ with coordinates $U,V,W,T$.
The arrangement of  coordinate planes
 $\{ UVWT= 0 \}$ in $\mathbb{P}^3$ intersects
 our plane $\mathcal{P}$ in four lines $\mathcal{H}_1,\mathcal{H}_2,
 \mathcal{H}_3,\mathcal{H}_4$.
 Here now is the promised result.

\begin{theorem}[Lie]\label{thm:lietetra}
Let $\mathcal{D}_1,\mathcal{D}_2$  be lines through $\mathbf{1}=(1:1:1:1)$ in $\mathcal{P}$.
 Then~$\mathcal{P}=\mathcal{D}_1\cdot \mathcal{D}_2$ if and only if the six lines $\,
 \mathcal{D}_1,\mathcal{D}_2,\,
 \mathcal{H}_1,\mathcal{H}_2,\mathcal{H}_3,\mathcal{H}_4\,$ are tangent to a common conic in $\,\mathcal{P}$.
\end{theorem}

 This result is featured in
\cite[p.~526]{LieGes2}. We here present a self-contained proof.

\begin{proof}
We identify $\mathcal{P}$ with the affine plane with coordinates $s$ and $t$ by setting
\begin{equation}
\label{eq:UVWT}
 U \,=\, 1 + a s + b t \,,\,\, \, V \,=\, 1 + c s + d t \,,\,\, \, W \,\,=\,\, 1+ e s+ f t \quad {\rm and} \quad T = 1. 
 \end{equation}
The origin $(s,t) = (0,0)$ corresponds to the distinguished point ${\bf 1}$. The two lines of interest are
 $\mathcal{D}_1 = \{s = 0\}$ and $\mathcal{D}_2 = \{t = 0\}$.
The four coordinate lines are
$\mathcal{H}_1 = \{U= 0\}$,
$\mathcal{H}_2 = \{V = 0\}$,
$\mathcal{H}_3 = \{W = 0\}$,
and $\mathcal{H}_4$ is the line at infinity in the $(s,t)$-plane.
Thus, the six scalars $a,b,c,d,e,f$ in (\ref{eq:UVWT}) specify the inclusions
 $\,\mathcal{D}_1, \mathcal{D}_2 \subset \mathcal{P} \subset \PP^3$.
With these conventions, $\,  \mathcal{D}_1,\mathcal{D}_2,\,
 \mathcal{H}_1,\mathcal{H}_2,\mathcal{H}_3,\mathcal{H}_4\,$ are tangent to a common conic 
 in $\mathcal{P}$ if and only if
\begin{equation}
\label{eq:det=0} {\rm det} \begin{pmatrix}
a & b & ab \\
c & d & cd \\
e & f & ef \end{pmatrix} \,\, = \,\, 0 . 
\end{equation}
The Hadamard product $\,\mathcal{D}_1 \cdot \mathcal{D}_2\,$ is a surface in $\PP^3$.
It has the parametric representation
\begin{equation}
\label{eq:UVWTmult}
\tilde U \,=\, (1+as)(1+bt) \,,\,\,\,
\tilde V \,=\, (1+cs)(1+dt) \,,\,\,\,
\tilde W \,=\, (1+es)(1+ft) \,,\,\,\,
\tilde T = 1 .
\end{equation}
This can be rewritten as
$$ \tilde U \,=\, U \,+\, ab \cdot st \, ,\,\quad
\tilde V \,=\, V \,+\, cd \cdot st \, ,\,\quad \tilde W \,=\, W \,+\, ef \cdot st \, ,\,\quad
\tilde T \,= \,T . $$
Hence the surface $\,\mathcal{D}_1 \cdot \mathcal{D}_2\,$
 equals the plane $\mathcal{P}$ in $\PP^3$
if and only if the point $(ab:cd:ef:0)$ lies in $\mathcal{P}$.
This happens if and only if the condition (\ref{eq:det=0}) holds.
Now the proof is complete.
\end{proof}

Given the tetrahedral theta surface (\ref{eq:TTS}),
we can now construct a one-dimensional family of pairs $\mathcal{C}_1,\mathcal{C}_2$
of generating curves.  The corresponding line pairs
$\mathcal{D}_1,\mathcal{D}_2$  in the plane (\ref{eq:log(TTS)})
are found as follows. We consider the one-dimensional family of
conics that are tangent to $\mathcal{H}_1,\mathcal{H}_2,\mathcal{H}_3,\mathcal{H}_4$.
Each such conic has tangent lines 
pass through $\mathbf{1}$. These are $\mathcal{D}_1$ and $\mathcal{D}_2$.

In the algebraic formulation above,
the geometric constraints can be solved as follows.
The given theta surface (\ref{eq:TTS})
is specified by any solution to $\alpha + \beta+\gamma = \delta$.
The desired one-dimensional family is
the solution set to five equations in the six unknowns 
$a,b,c,d,e,f$. In order for the planes in
 (\ref{eq:UVWT}) and (\ref{eq:log(TTS)}) to agree, we need
$\,  \alpha a + \beta c + \gamma e \,=\, \alpha b + \beta d + \gamma f \,=\, 0 $.
To get unique parameters for our lines, we may also fix
 $a$ and $f $ in $\CC$. Finally,  the
quadratic equation (\ref{eq:det=0}) must be satisfied.
These five constraints define a curve in $\CC^6$
whose points are the solutions $\,(\mathcal{D}_1,\mathcal{D}_2)\,$ 
to $\,\mathcal{D}_1 \cdot \mathcal{D}_2 = \mathcal{P}\,$ 
and hence the solutions $\,(\mathcal{C}_1,\mathcal{C}_2)\,$
to $\,\mathcal{C}_1 + \mathcal{C}_2 = \mathcal{S}$.

We demonstrate this algorithm for computing generating curves of (\ref{eq:TTS})
in an example.

\begin{example} \label{ex:werevisit}
	We revisit Example \ref{eq:amoebaofline} and the corresponding tetrahedral theta surface given
	 \eqref{eq:fourterms}.	After replacing $Z$ with $Z+\log(3)$, the resulting surface
	 passes through $\mathbf{0}$. We have
$$ \begin{matrix}
	\mathcal{S} & = & \,\,\{ \,\exp(X)\,+\,\exp(Y)\,-\,3\exp(Z)\,+\,1\,\,=\,\,0 \,\} & \subset \,\, \CC^3, \smallskip \\
	\mathcal{P} &  = & \{ \,U \, + \,V \, - \, 3 W \, + \, T \,\, = \,\,0 \, \}  & \subset \,\, \PP^3. 
\end{matrix}
$$	
To find a valid parametrization of $\mathcal{P}$ as in (\ref{eq:UVWTmult}),
we consider the equations in $a,b,c,d$, $e,f$ described above.
We fix $a=1 , f =0$ and we leave $b$ unspecified. 
The remaining parameters are determined as $\, c = 1$,  $\,d = -b$, $\, e = 2/3$,
by requiring
(\ref{eq:det=0}) and that (\ref{eq:UVWTmult}) lies on $\mathcal{P}$.

We conclude that the plane $\mathcal{P}$ has the parametrizations
$$
\tilde U \,=\, (1+s)\cdot (1+bt)\,, \, 
\tilde V \,=\, (1+s)\cdot(1-bt)\,, \,
\tilde W \,=\, \left(1+\frac{2}{3}s \right)\cdot1, \,
 \hbox{for all}\,\, \, \quad b \in \CC \backslash \{0\}.
$$
Our tetrahedral theta surface $\mathcal{S}$ has 
the one-dimensional family of parametrizations:
$$ X \,=\, {\rm log}(1+s) \,+\, {\rm log}(1+bt)\,, \quad
Y \,=\, {\rm log}(1+s) \,+\, {\rm log}(1-bt)\,, \quad
Z \,=\, {\rm log} \left( 1 + \frac{2}{3} s \right) .$$
This example is admittedly quite special, but the method
works for all tetrahedral theta surfaces, i.e.~whenever the quartic curve $\mathcal{Q}$ is
among the last three types in Figures \ref{fig:nodalquartics} 
and \ref{fig:graphs}.
\end{example}

We conclude this article by returning to the twelve plaster models
of theta surfaces constructed by the doctoral students of Sophus Lie at Leipzig
in 1892. In Figure \ref{fig:KummerWiegner} we display
one model due to  Richard Kummer \cite{Kum}
and one model due to Georg Wiegner \cite{Wie}.

\begin{figure}[h!]
   \begin{minipage}{0.50\textwidth}
     \centering
     \includegraphics[width=.99\linewidth]{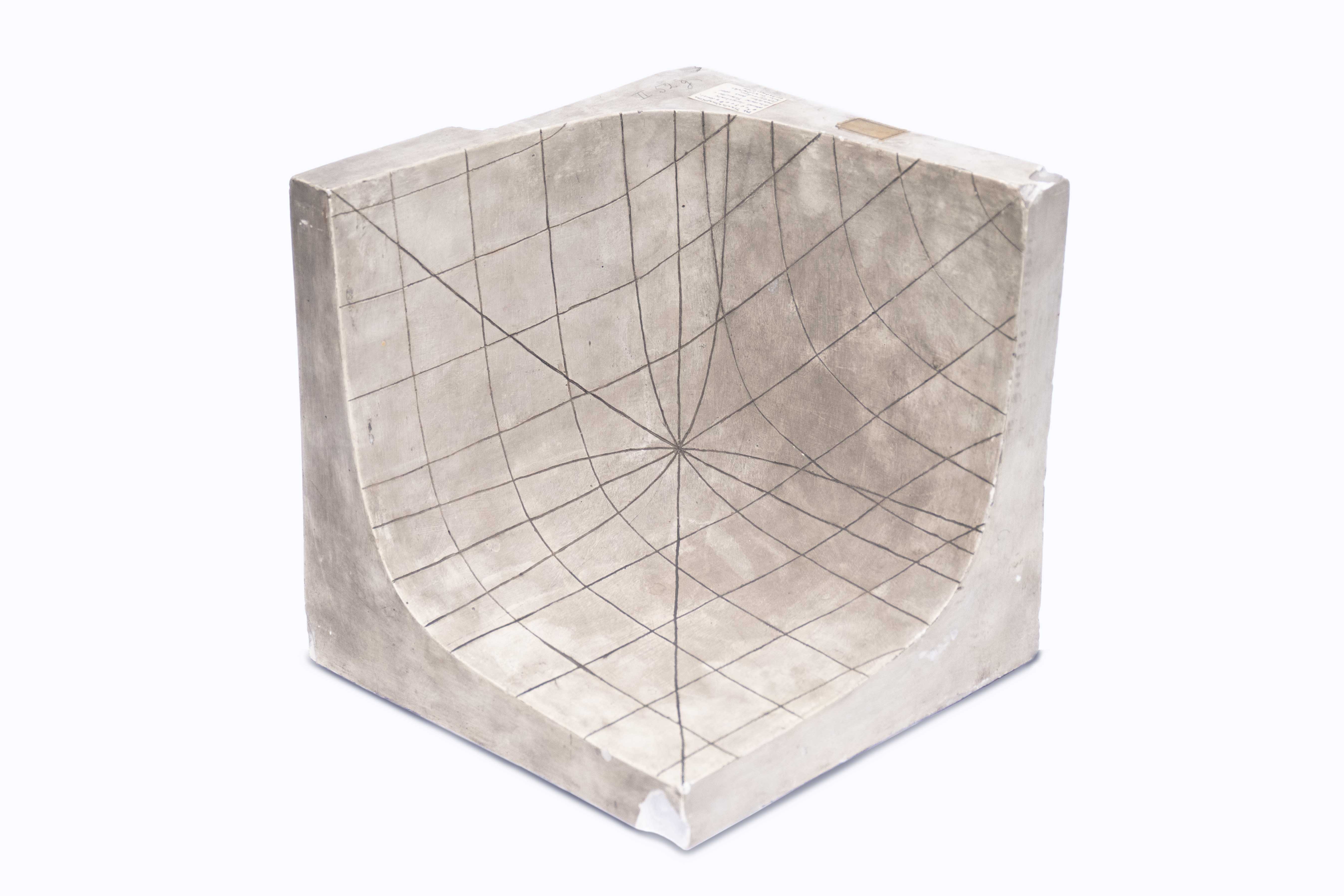} \\
   \end{minipage}\hfill
   \begin{minipage}{0.50\textwidth}
     \centering
     \includegraphics[width=.98\linewidth]{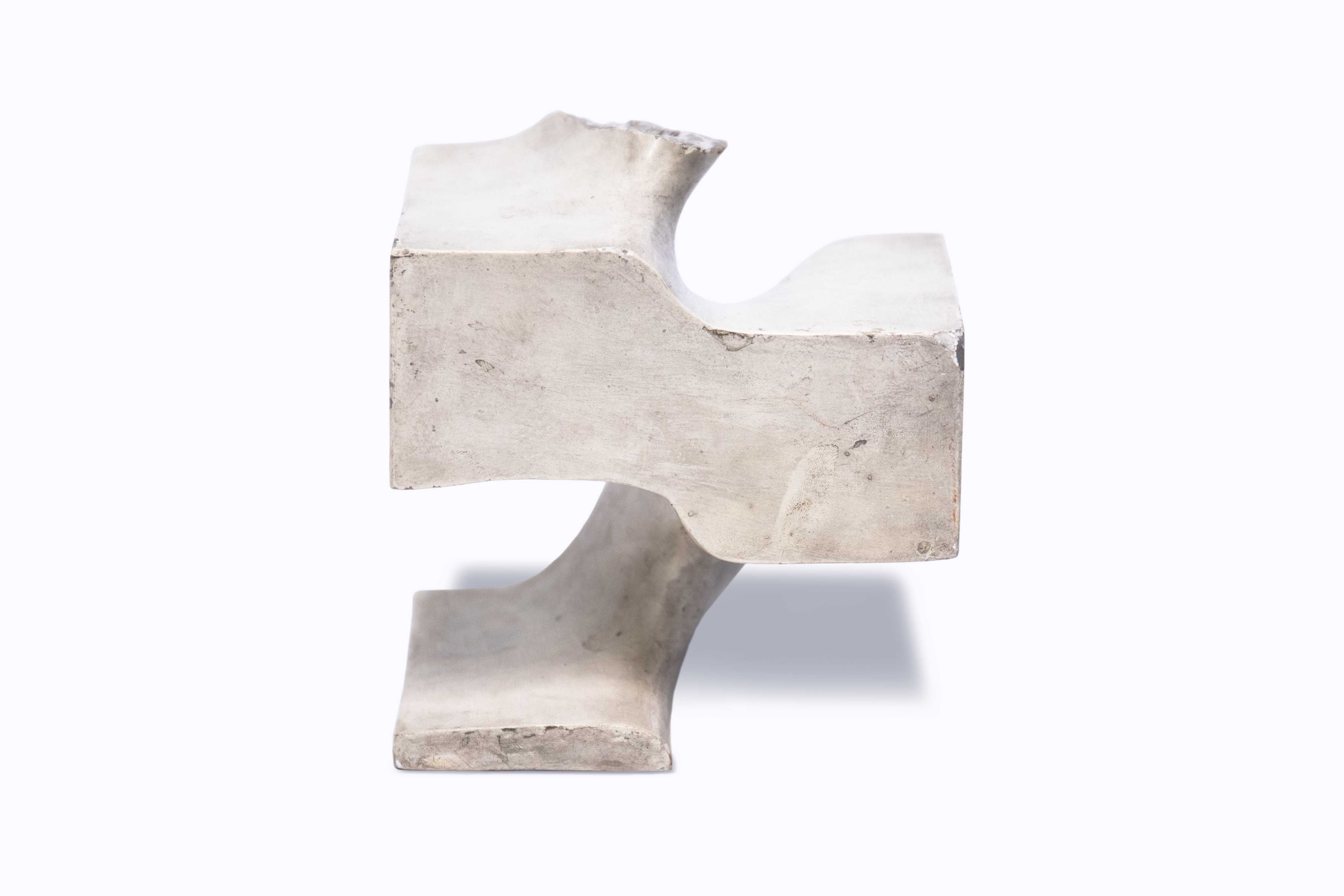} \\
   \end{minipage}
   \caption{Plaster models of theta surfaces
constructed  in the 1890s by Lie's students
 Kummer (left) and Wiegner (right).
These models are still in the collection at Universit\"at Leipzig.
   \label{fig:KummerWiegner}
   }
\vspace{-0.02in}   
\end{figure}

The model on the left in Figure \ref{fig:KummerWiegner} shows
the tetrahedral theta surface
$$ \mathcal{S} \quad = \quad \bigl\{
\, 10^{-X} + 10^{-Y} + 10^{-Z} \, = \, 1 \,\bigr\}. $$
Kummer derives this surface in \cite[Section III.6]{Kum}
from a pencil of conics like in Example~\ref{eq:amoebaofline}.
In \cite[p.~32]{Kum} he applies a particular
transformation to the surface, which seems to be advantageous for the practical construction of a plaster model.
The one-dimensional family of Minkowski decompositions $\,\mathcal{S} = \mathcal{C}_1 + \mathcal{C}_2\,$
into curves can be found using our algorithm for Theorem \ref{thm:lietetra}. The model on the right in Figure \ref{fig:KummerWiegner} shows
another theta surface, namely
$$ \mathcal{S} \quad = \quad \biggl\{\,
{\rm tan}(Z) \, + \, \frac{2X}{X^2+2Y} \,\, = \,\,0 \,\biggr\}. $$
Wiegner derives this equation in \cite[Section IV.11]{Wie} from a quartic $\mathcal{Q}$
that decomposes into a cubic curve and one of its flex lines.
 In \cite[Section II.4]{Wie}, Wiegner rederives the Weierstrass normal form,
 and he fixes the flex line to be the line at infinity.
For the surface $\mathcal{S}$ he starts with
the rational cubic $\,q = y^2 - x^2(x-1)$, and he ends up
on \cite[p. 65]{Wie} with the equation seen above. The surface $\mathcal{S}$ is shown in \cite[Figure II, Tafel A]{Wie}.
In his appendix  \cite[p.~82]{Wie}, Wiegner offers a delightful
description of how one actually builds a plaster model in practice.

\smallskip

This final section connects the 19th century with
the 21st century, and differential geometry with
algebraic geometry. Theta surfaces are beautiful objects,
not just for 3D printing, but they offer new vistas
on the moduli space of genus $3$ curves.
The explicit degenerations in Sections \ref{sec5}
and \ref{sec6}, and the tools from
numerical algebraic geometry in Section~\ref{sec4},
 should be useful for many applications,
such as three-phase solutions of the KP equation~\cite{DFS}.

\medskip

\textbf{Acknowledgments:} We are thankful to the referees for their helpful comments and remarks. We are grateful to John Little for informing us of the quadric double translation surfaces of Example \ref{Quadric} and to Evgeny Ferapontov for mentioning his paper \cite{BalFer} to us.  

\bibliographystyle{spmpsci}

\end{document}